\documentclass[12pt,a5paper]{article}

\usepackage{euler,amsmath,amsfonts,defns,url,graphicx}
\usepackage{ajr}
\allowdisplaybreaks

\newcommand{\RR}{\mathbb{R}}

\newcommand{\cM}{\mathcal{M}}
\newcommand{\cL}{\mathcal{L}}

\newcommand{\npde}{\textsc{npde}}
\newcommand{\ibc}{\textsc{cc}}
\newcommand{\nsm}{\textsc{nsm}}
\newcommand{\ncm}{\textsc{ncm}}

\newcommand{\inti}{\int_{-\infty}}

\newcommand{\up}{U_{j+1}}
\newcommand{\um}{U_{j-1}}

\newcommand{\uj}{U_{j}}
\newcommand{\duj}{\D t{U_{j}}}
\newcommand{\xj}{X_{j}}
\newcommand{\xpm}{X_{j\pm 1}}
\newcommand{\pars}{\vec\epsilon} 
\newcommand{\cosin}{\operatorname{csn}}
\newcommand{\Z}[1]{\operatorname{\mathcal{H}}_{#1}}

\newcommand{\mud}{\mu\delta}

\newcommand{\bspde}{Burgers' \npde~\eqref{eq:burg}}
\newcommand{\npdee}{\npde~\eqref{eq:spde}}
\newcommand{\eigf}{e} 
\newcommand{\vecs}[1]{\mathbb{#1}}
\def\arcproject#1{}\def\arcstatement#1{}

\usepackage{pdfcomment}

\usepackage{pgfplots} 
\IfFileExists{ajr.sty}{}{\pgfplotsset{compat=1.8}}

\title{Resolution of subgrid microscale interactions enhances the discretisation of nonautonomous partial differential equations}

\author{J.~E. Bunder\thanks{School of Mathematical Sciences, University of Adelaide, South Australia~5005, Australia.  
\protect\url{mailto:judith.bunder@adelaide.edu.au}} 
\and A.~J. Roberts\thanks{School of Mathematical Sciences, University of Adelaide, South Australia~5005, Australia.  
\protect\url{mailto:anthony.roberts@adelaide.edu.au}}}

\date{\today}

\begin{document}

\maketitle

\begin{abstract}
Coarse grained, macroscale, spatial discretisations of nonlinear nonautonomous partial differential\slash difference equations are given novel support by centre manifold theory.
Dividing the physical domain into overlapping macroscale elements empowers the approach to resolve significant subgrid microscale structures and interactions between neighbouring elements.
The crucial aspect of this approach is that centre manifold theory organises the resolution of the detailed subgrid microscale structure interacting via the nonlinear dynamics within and between neighbouring elements.
The techniques and theory developed here may be applied to soundly discretise on a macroscale many dissipative nonautonomous partial differential\slash difference equations, such as the forced Burgers' equation, adopted here as an illustrative example. 
\end{abstract}


\tableofcontents

\section{Introduction}

This article develops a systematic approach to constructing spatially discrete macroscale models of nonlinear nonautonomous partial differential\slash difference equations (collectively denoted nonlinear \npde{}s).
In particular, we address issues arising from rich and significant microstructure on a considerably smaller spatial scale than the macroscale grid.
Although this article concentrates on mapping nonautonomous dynamics defined on a continuum or fine microscale grid onto a coarse macroscale grid, the theory permits a potentially multigrid hierarchy. 
Thus the scope of this research is not only a rescaled discretisation of subgrid dynamics, but a fully multiscale theory which enhances our understanding of how fluctuation information is transferred over multiple levels of scale---identified
by Dolbow et al.~\cite{Dolbow04} as an outstanding unresolved problem. 

We invoke nonautonomous centre manifold (\ncm)  theory~\cite[e.g.]{Aulbach2000, Potzsche2006, Haragus2011}  to ensure the accuracy, stability and efficiency of relatively coarse spatial discretisations of \npde{}s. 
An understanding of centre manifold theory is best obtained via normal form theory, developed for both random~\cite{Arnold03} and deterministic normal forms~\cite[e.g.]{Murdock03}:  
a coordinate transform separates the dynamics of a given system into coarse grid slow dynamics and fast subgrid dynamics; 
in the new coordinates the subgrid modes decay exponentially quickly over a finite domain from all initial conditions; 
the slow modes evolve on the centre manifold as a rigorous closure on the coarse grid.  
Section~\ref{sec:cm} discusses the theory which ensures such a slow centre manifold both exists and emerges exponentially quickly as a macroscale discretisation.
However, complications arise from nonautonomous effects and the practical challenge is to ensure the intricate spatio-temporal dynamics of the given \npde\ are reflected in the closed form nonautonomous slow manifold (\nsm{}) defined on the coarse macroscale grid.


Earlier research developed a dynamical systems approach to discretising autonomous deterministic \pde{}s.
Examples include the fourth order Kuromoto--Sivashinsky equation~\cite{MacKenzie00a, MacKenzie05a}, the two dimensional Ginzburg--Landau equation~\cite{Roberts2011a} and Burgers' equation~\cite{Roberts98a, Roberts01a, Roberts01b}.
Section~\ref{sec:cm} significantly extends theoretical support for this dynamical systems approach to the discrete multiscale  modelling of general nonlinear nonautonomous \pde{}s.

The scope of this article is to  develop a method for obtaining a macroscale closure of a  field~$u(x,t)$ evolving in continuous time~$t$ in some Banach space of one spatial dimension, as defined by the general nonlinear \npde{}
\begin{equation}
    \D tu=\cL(u) u+\alpha f(u)+\varepsilon\phi(u,x,t) ,
    \label{eq:spde}
\end{equation}
where $\cL(u)$~is a smooth, second-order, quasi-linear operator which is dissipative (see Assumption~\ref{ass:sss}), $f(u)$~denotes smooth, autonomous, perturbations, and the `forcing' term $\phi$~represents the nonautonomous effects in the \npde.
We expect this method to generalise to higher spatial dimensions as it does for autonomous \pde{}s~\cite{Roberts2011a}.
In principle the nonautonomous forcing may be either deterministic or nondeterministic (such as stochastic noise), but here we concentrate on the deterministic case. 
When necessary for definite theoretical statements and for displayed numerical simulations, we adopt boundary conditions for the \npdee\  of $L$-periodicity, $u(x+L,t)=u(x,t)$, and assume that initial conditions lie within the finite domain of validity of the \ncm\ theory.
Section~\ref{sec:cm} discusses that the \npdee\ not only includes the class of \npde{}s differential in space, but also cases where the spatial domain~$x$ is discrete and the \npde{}s are difference equations in space such that~$\cL$, $f$ and~$\phi$ are discrete operators on the spatial domain. 

Typical issues and results are illustrated herein via the definite example of discretising in space the nonlinear Burgers' equation with nonautonomous forcing:
\begin{equation}
    \D tu +\alpha u\D xu =\DD  x u +\varepsilon\phi(x,t)
    \label{eq:burg}\,.
\end{equation}
The simulation of the \bspde\ in Figure~\ref{fig:micro} illustrates complicated microscale fluctuations and their cumulative appearance in the macroscale.
Givon et al.~\cite[p.R58]{Givon04} similarly used specific example problems to develop and illustrate many issues in their approach. 

The issue of macroscale modelling closure is a longstanding problem in autonomous dynamics, let alone in cases with general nonautonomous dynamics.
For example,  it is not at all clear how best to discretise the nonlinear advection term~$\alpha uu_x$ in \bspde{}. 
Should one use $\rat12\alpha(u^2)_x$ or should $u$~and~$u_x$ be discretised independently? Perhaps a combination of these two options is beneficial, as suggested by Fornberg~\cite{Fornberg73} and used to improve stability~\cite[e.g.]{Foias91}.
The traditional approach for linear systems of discretising each term separately does not tell us how to proceed. 
Yet a centre manifold approach supports a definite scheme for best performance on macroscale grids~\cite{Roberts98a}.

Sections \ref{sec:quad}~and~\ref{sec:sto} construct, as an example application, the following low order, discrete macroscale model \npde{}s for the nonlinear dynamics of the forced \bspde:
\begin{subequations}\label{eq:lowg}%
\begin{align}
    \duj\approx{}&
    \frac1{H^2}\big( 1+\rat1{12}\alpha^2H^2\uj^2 \big)(\up-2\uj+\um)
    -\alpha\frac{1}{2H}\uj(\up-\um)
    \label{eq:lowga}
   \\{}&
    +\varepsilon\left[ \phi_{j,0} -\alpha\frac{2H}{\pi^2}\phi_{j,1}\uj
    -\alpha^2\frac{8H^2}{3\pi^4}\phi_{j,2}\uj^2 \right]
    + .01643\alpha^2H^2\varepsilon^2\uj\,,
    \label{eq:lowgb}
\end{align}%
\end{subequations}
where grid values \(\uj(t)\approx u(X_j,t)\) on a grid~\(X_j\) of macroscale spacing~\(H\).
This model is for the case when the subgrid microscale structures within each element is truncated to the first three Fourier modes:
\begin{equation}
    \phi(x,t)=\phi_{j,0}(t) +\phi_{j,1}(t)\sin[\pi(x-\xj)/H]
    +\phi_{j,2}(t)\cos[2\pi(x-\xj)/H]\,.\label{eq:f}
\end{equation}
The terms~\eqref{eq:lowga} of the discretisation~\eqref{eq:lowg} is the so-called holistic discretisation for Burgers' equation which has good properties on finite sized elements~\cite{Roberts98a}; in particular, arising from nonlinear advection-diffusion interaction, the nonlinearly enhanced diffusion promotes stability of the scheme for larger nonlinearity. 
The terms~\eqref{eq:lowgb} of the discretisation approximate some of the  influences of the subgrid microscale nonautonomous forcing.
The nonlinearity in the subgrid microscale dynamics of Burgers' equation transforms the additive nonautonomous components~\eqref{eq:f} of the \bspde\ into multiplicative components in the discretisation, such as $\phi_{j,1}\uj$.  
Many modelling schemes miss such multiplicative terms because they do not resolve the subgrid microscale processes that are revealed by the \ncm\ methodology.

\begin{figure}
    \centering
\includegraphics{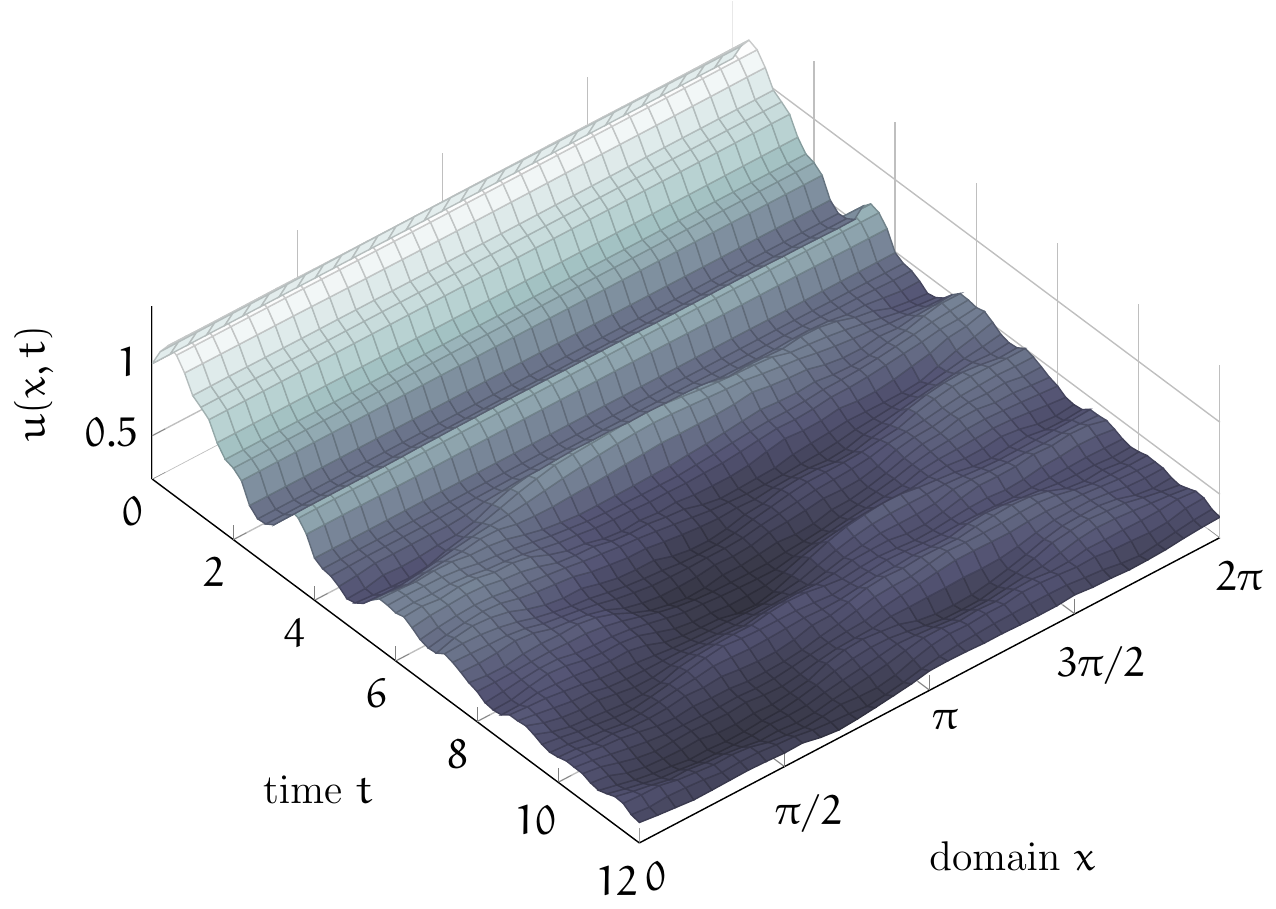}
\caption{An example microscale simulation of the spatially discretised forced \bspde, with initial condition $u(x,0)=1$, nonlinearity $\alpha=0.3$ and forcing coefficient $\varepsilon=0.05$.  
The nonautonomous forcing $\phi(x,t)=\xi_i(t)$, for each fine space mesh index $i=x/\Delta x$\,, is defined by the  Lorenz equations: $d\xi_i/dt=10(\eta_i-\xi_i)$, $d\eta_i/dt=\xi_i(28-\zeta_i)-\eta_i$, $d\zeta_i/dt=\xi_i\eta_i-\rat83\zeta_i$\,, and random initial values \((5,8,N(10,1))\) 
(so $\varepsilon\phi(x,0)$ has mean~$0.5$).
This simulation uses a fine space-time mesh with $\Delta x=\pi/16$ and $\Delta t=0.01$ but plotted every $19$th~time step. As time progresses, fluctuations induced by the nonautonomous forcing~$\phi(x,t)$ extend across the spatial domain via nonlinear interactions.}
    \label{fig:micro}
\end{figure}

As seen in Figure~\ref{fig:micro}, subgrid variations in the forcing typically generate microscale spatial structures with relatively high wavenumber and steep variations.
Through a form of resonance, the mean effects of these modes may be important on the large scale dynamics as quantified by Section~\ref{sec:sto}.
For example, the last term of the discrete model \npde~\eqref{eq:lowg}, being proportional to~$\alpha^2\varepsilon^2\uj$, arises from nonlinear self-interactions of subgrid scale dynamics.
Herein the term `resonance' includes phenomena where nonautonomous fluctuations interact with each other and themselves through nonlinearity in the dynamical system to generate not only long time drifts but also potentially to change stability as is also recognised in stochastic systems~\cite[e.g.]{Knobloch83, Boxler89, Drolet01, Roberts03c, VandenEijnden05c}.
Coarse grained discrete model \npde{}s, such as \eqref{eq:lowg}, which use large space-time grids for efficiency, must account for such subgrid microscale resonance in their closure in order to resolve the significant yet subtle interactions.

Another example is furnished by coarse grid modelling of the nonlinear nonautonomous Burgers'-like microscale grid equation
\begin{equation}
\frac{du_i}{dt}=\frac{4}{H^2}(u_{i+1}-2u_i+u_{i-1})
-\frac\alpha H u_i(u_{i+1}-u_{i-1}) +\varepsilon\phi_i(t) ,
\label{eq:bssde}
\end{equation}
where subscript~$i$ indexes the fine grid (in contrast to~$j$ which we use to index a coarse macroscale grid), the first term on the right-hand side is a diffusive-like dissipative reaction with neighbours, the second term represents a form of nonlinear advection of strength~$\alpha$, and the last term is an additive nonautonomous forcing which is independent at each grid point~$i$ but of uniform strength~$\varepsilon$.  
The theory described in Section~\ref{sec:cm} for the general \npdee\ also supports the modelling of discrete models such as~\eqref{eq:bssde} on a coarser grid with twice the grid spacing.
Extending earlier research for a discrete linear autonomous model~\cite{Roberts08c}, define the coarse grained amplitudes~$\uj(t)=u_{2j}(t)$ and our coarse grid model of~\eqref{eq:bssde} is
\begin{align}
\duj\approx {}&\frac{1}{H^2}(\up-2\uj+\um)
-\frac\alpha {2H}\uj(\up-\um) \nonumber\\
{}&+\varepsilon\left[\psi_{j0}(t)
-\frac{\alpha H}8\uj\psi_{j1}(t)\right].
\label{eq:bssdm}
\end{align}
To this level of approximation the coarse grid evolution~\eqref{eq:bssdm} is  the same form as the fine grid~\eqref{eq:bssde} with appropriately renormalised diffusion and nonlinear advection, but with the nonautonomous forcing  now weakly correlated across coarse grid points as the coarse grid nonautonomous forcing is the multigrid restriction $\psi_{j0}=\frac14\phi_{2j-1}+\frac12\phi_{2j}+\frac14\phi_{2j+1}$\,, and with a new multiplicative term arising from resolving the subgrid interaction between the nonlinear advection and structure in the subgrid microscale.
To complement other methods, \ncm\ theory~\cite[e.g.]{Aulbach2000, Potzsche2006, Haragus2011} provides a systematic approach to determining the macroscale modelling of such microscale nonautonomous nonlinear interactions.

Homogenisation and averaging are popular methods for deriving mean effects on slow macroscale modes from fast microcale modes~\cite[e.g.]{Pavliotis07}.
However, homogenisation requires a small parameter~$\epsilon$ which measures an asymptotically infinite scale separation between slow and fast modes. 
In discretising the \npdee\ with support from \ncm\ the scale separation between these slow and fast modes may be finite, which is unsuitable for homogenisation. 
Even more importantly, when a macroscale discretisation of a discrete microscale problem is required, such as \eqref{eq:bssde}$\mapsto$\eqref{eq:bssdm},  there is no definable asymptotically small scale separation parameter~$\epsilon$ and homogenisation is not applicable.
Dolbow et al.~\cite[p.30]{Dolbow04} specifically call for ``new multiscale mathematical methods developed and used to derive multiscale models for some of the `difficult' cases in multiscale science; e.g., problems without strong scale separation,''.  
Our application of \ncm\ provides such a methodology.


Section~\ref{sec:quad} provides one resolution of the mean effects of detailed linear and nonlinear subgrid microscale structure on the macroscale grid values~$\uj$, with support provided by \ncm{}.
Section~\ref{sec:onhnni} discusses a case of \bspde\ where traditional straightforward numerical approximations miss all of the easily apparent effects of the subgrid microscale structure on the macroscale discretisation and are thus incapable of finding an accurate macroscale closure for \npde{}s of this type.
The wide ranging reports of Dolbow et al.~\cite{Dolbow04} and Brown et al.~\cite{Brown08} identify accurate macroscale closure as an outstanding challenge in multiscale modelling of nonautonomous physical systems of great interest to applied mathematics.
We further develop and validate the \ncm\ methodology for closure of macroscale discretisations.

The focus here is on deterministic forcing~$\phi(x,t)$. 
Nonetheless, most of the analysis and models in Sections~\ref{sec:cm}, \ref{sec:quad} and~\ref{sec:sto} also hold for nondeterministic stochastic forcing.  
In cases where the forcing~$\phi(x,t)$ is stochastic it is interpreted in the Stratonovich sense so that the rules of traditional calculus apply as preferred by many physicists and engineers. 

\section{Nonautonomous centre manifold theory underpins modelling}
\label{sec:cm}

This section details one way to place the spatial discretisation of \npde{}s within the purview of \ncm\ theory.
Powerful and general \ncm\ theory was developed by Aulbach and Wanner~\cite{Aulbach96, Aulbach99, Aulbach2000}: 
first they proved the existence of centre and other integral manifolds%
\footnote{In nonautonomous theory the term ``integral manifold'' is usually used to denote the generalisation of the concept of an invariant manifold in autonomous systems.} 
of `infinite' dimensional systems~\cite[Theorem~6.1]{Aulbach96}; 
second they established the emergence of solutions on the centre manifold by proving topological equivalence to a system with corresponding linear hyperbolic modes and no hyperbolic dependence in the centre variables~\cite[Theorem~4.1]{Aulbach2000}.
Subsequently, P\"otzche and Rasmussen~\cite[Proposition~3.6]{Potzsche2006} established a useful approximation theorem.
Although Aulbach and Wanner's global theory is limited to Lipschitz nonlinearities, the theory is straightforwardly extended to more general nonlinearities using cut-off functions at the cost of then being restricted to a finite domain in state space, as discussed by Haragus and Iooss~\cite[Chapt.~2]{Haragus2011} for example.
In the stochastic case, the centre manifold theory of Boxler~\cite{Boxler89} was elaborated by Arnold~\cite[Chapt.~7]{Arnold03}, and a freely available concise summary is available~\cite[Appendix~A, e.g.]{Roberts05c}.
Subject to some conditions,  \ncm\ theory assures us of the existence and emergence of discrete models of the general nonlinear  \npdee.

To proceed, we embed the dynamics of the \npdee\ into a higher dimensional nonautonomous system.
Then from the base of a subspace of equilibria, this section establishes the existence and emergence of a nonautonomous slow manifold (\nsm{}) that forms the macroscale discrete model.
The key here is simply to establish the preconditions for the application of \ncm\ theory:
when a definite example is discussed we invoke the spatially continuous \bspde, or the discrete microscale dynamics of~\eqref{eq:bssde}.
Then the established \ncm\ theory rigorously supports a wide range of models and illuminates a wide range of modelling issues---to know that useful results are obtainable we only need to verify the preconditions.

\subsection{Divide space into overlapping finite elements}
\label{sec:dsiofe}

The method of lines discretises a \pde\ in space~$x$ and integrates in time as a set of ordinary differential equations---sometimes called a semi-discrete scheme \cite[e.g.]{Foias91, Foias91b}.
Similarly, this article only discusses the relatively coarse grained spatial discretisation of \npde{}s as a continuous time,  dynamical system.

\begin{figure}
\centering
\begin{tabular}{cc}
(a) &
\begin{tikzpicture} 
\begin{axis}[xlabel={$x$}, axis equal image, no marks
,axis x line=middle , hide y axis ,samples=41 , thick 
,ymin=0,xmin=-2.4,xmax=2.4
,xtick={-2,...,2}
,xticklabels={$X_{j-2}$,$X_{j-1}$,$X_{j}$,$X_{j+1}$,$X_{j+2}$}
, smooth, tension=0.8]
\addplot coordinates {(-2,0.8)(-1.5,0.6)(-1,0.8)(-0.5,1)(0,0.8)};
\addplot coordinates {(-1,1)(-0.5,1.1)(0,1)(0.5,0.9)(1,1)};
\addplot coordinates {(0,0.6)(0.5,0.8)(1,0.6)(1.5,0.5)(2,0.6)};
\end{axis} 
\end{tikzpicture}
\\
(b)&
\begin{tikzpicture} 
\begin{axis}[xlabel={$x$}, axis equal image
,axis x line=middle , hide y axis ,samples=41 , thick 
,ymin=0,xmin=-2.4,xmax=2.4
,xtick={-2,...,2}
,xticklabels={$X_{j-2}$,$X_{j-1}$,$X_{j}$,$X_{j+1}$,$X_{j+2}$}
, smooth, tension=0.8]
\addplot coordinates {(-2,0.2)(-1,1)(0,1)};
\addplot coordinates {(-1,1)(0,1)(1,0.5)};
\addplot coordinates {(0,1)(1,0.5)(2,0.5)};
\end{axis} 
\end{tikzpicture}
\end{tabular}
\caption{Construct three overlapping elements centred about the macroscale grid points~$X_j$ and~$X_{j\pm1}$  with (a)~decoupled elements $\gamma=0$ and (b)~fully coupled elements $\gamma=1$. 
The three curves are the microscale field solutions within each overlapping element, $u_{j-1}(x,t)$,  $u_{j}(x,t)$ and $u_{j+1}(x,t)$. 
When decoupled, the microscale fields solutions within the three elements are independent of each other. 
At full coupling the \ibc~\eqref{eq:ibc} ensure that $u_j(x,t)=u_{j-1}(x,t)$ at the macroscale grid points~$X_j$ and~$X_{j-1}$, and $u_j(x,t)=u_{j+1}(x,t)$ at the macroscale grid points~$X_j$ and~$X_{j+1}$.}
\label{fig:elements}
\end{figure}
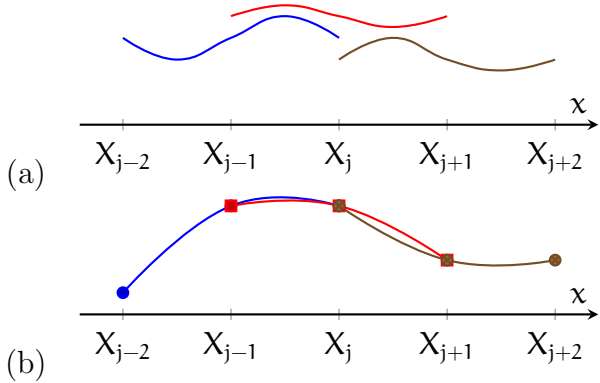

Place the spatial discretisation of a nonlinear \npdee, such as the \bspde, within the purview of \ncm\ theory by the following artifice.
Let equi-spaced macroscale grid points at~$\xj$ be a distance~$H$ apart.
Define the $j$th~element as $|x-\xj|\leq H$ which overlaps with neighbouring elements.  
Therefore, the $j$th~element is centred about $x=X_j$ and has a left and right boundary at $x=X_{j-1}$ and $x=X_{j+1}$, respectively. 
Say there are $m$~such elements on an \(L\)-periodic domain.
The first step is to embed the dynamics of the \npdee\ into a higher dimensional system defined with these overlapping elements.
Let $u_j(x,t)$~denote the microscale field in the $j$th~element $X_{j-1}\leq x\leq X_{j+1}$\,. 
Then on each element solve the \npdee\ with the field in the $j$th~element, namely we analyse the set of \npde{}s,
\begin{equation}
    \D t{u_j}=\cL(u_j) u_j+\alpha f(u_j)+\varepsilon\phi_j(u_j,x,t).
    \label{eq:spdei}
\end{equation}
Following the analogous and proven  approach to discretising autonomous \pde{}s~\cite[e.g.]{MacKenzie05a, Roberts00a, Roberts2011a}, we invoke nonlocal coupling conditions~(\ibc{}s) between neighbouring elements:
\begin{equation}
    u_j(\xpm,t) = (1-\gamma)u_j(\xj,t) + \gamma u_{j\pm1}(\xpm,t).
    \label{eq:ibc}
\end{equation}
The coupling parameter~$\gamma$ controls the flow of information between adjacent elements: 
when $\gamma=0$\,, adjacent elements are decoupled by the \ibc~\eqref{eq:ibc} as illustrated by Figure~\ref{fig:elements}(a); 
whereas when $\gamma=1$\,, the \ibc~\eqref{eq:ibc} requires the microscale field in the $j$th~element at the grid points~$X_{j\pm1}$ to extrapolate to the microscale fields in the neighbouring  $(j\pm1)$th~elements at~$X_{j\pm1}$ as illustrated by Figure~\ref{fig:elements}(b).
Such coupling of overlapping elements is analogous to the `border regions' of the heterogeneous multiscale method~\cite[e.g.]{E04}, to the `buffers' of the gap-tooth scheme~\cite[e.g.]{Samaey03b}, and to the overlapping domain decomposition that improves convergence in waveform relaxation of parabolic \textsc{pde}s~\cite[e.g.]{Gander98}.
The specific coupling \ibc{}~\eqref{eq:ibc} ensures discrete models are consistent with deterministic \pde{}s to high order in~$H$ as the element size~$H\to0$\,, both linearly~\cite{Roberts00a} and nonlinearly, and also in multiple space dimensions~\cite[\S4]{Roberts2011a}.

The macroscale model is a system of coupled differential equations for the  grid values \(U_j(t):=u_j(X_j,t)\) in each element. 


When necessary for definitive theory and for numerical simulations, define $m$~overlapping elements with $j=1,2,\ldots, m$\,, $L$-periodic boundary conditions ($L=mH$) for the global field~$u(x,t)$, and require that the fields in the $m$~elements satisfy $u_{j\pm m}(x\pm L,t)=u_j(x,t)$ for all~$x$ and elements~$j$.   
Equivalently, consider the \npde\ in the interior of a domain which is sufficiently large such that the physical boundaries are far enough away to be immaterial.
Evidently, the set of \npde{}s~\eqref{eq:spdei} form a higher dimensional nonautonomous system that effectively reduces to the \npdee\ in the limit of full coupling.

We avoid defining a precise function space for analysis of the \npde{}s~\eqref{eq:spdei}, although \(u\), and hence~\(u_j\), must be in a suitable Banach space.
The reason is that extant theorems place differing conditions on the nature of the functions appearing in the \npde; see Corollary~\ref{lem:exist} for examples.
Herein we primarily develop a formal methodology applicable to all established rigorous theoretical conditions, but flexible enough to cater also for a wide class of physically interesting nonautonomous systems.    
Depending upon the details of an application, the theoretical assurances changes as indicated by alternative conditions. 

\paragraph{Linearise about useful equilibria}
The second step for the macroscale modelling is to  anchor the discrete modelling, defined by the overlapping elements, upon the subspace of equilibria~$\vecs E_0$ of piecewise constant solutions. 
The equilibria are that $u_j(x,t)=\uj ={}$constant for all~$j$, for the autonomous case $\varepsilon=0$\,, for no interelement coupling $\gamma=0$ in the \ibc~\eqref{eq:ibc}, and no nonlinearity $\alpha=0$\,.
For each of the equilibria in~$\vecs E_0$ it is straightforward to find the linear Oseledec spaces that form the foundation of the \nsm\ as here they are then standard linear eigenspaces.
Linearising about each equilibria, $u_j(x,t)=U_j+u'_j(x,t)$ for small perturbations~$u'_j(x,t)$, the dynamics of the \npdee\ with  \ibc~\eqref{eq:ibc} reduce to that of dissipation within each element isolated from all neighbours:
\begin{equation}
    \D t{u'_j}=\cL_j u'_j  
    \quad\text{such that}\quad
    u'_j(\xpm,t)=u'_j(\xj,t)\,,
    \label{eq:spdel}    
\end{equation}
where the prime denotes a perturbation to the equilibrium field of each element of~$\vecs E_0$ and $\cL_j=\cL(\uj )$;
for example, \bspde\ linearises to the diffusion equation $u'_t=u'_{xx}$\,.%
\footnote{The need to linearise the \npde{} places restrictions on the form of the quasi-linear operator $\cL(u)$. 
Specifically, we cannot consider cases where $\cL_j=\cL(U_j)=0$ since on linearisation we do not obtain the desired linear form but instead the trivial \pde{} $u_t'=0$ which is non-dissipative in the neighbourhood of the equilibria~$\vecs E_0$ and \ncm\ theory is not applicable. 
For example, we cannot consider $\cL(u)=u_x\partial^2_{x}$ as $u_x=0$ when $u=U_j$.}

\begin{assumption}[slow+stable spectrum] \label{ass:sss}
Assume the linearised system~\eqref{eq:spdel} has dissipative dynamics on each of the $m$~elements.
More specifically, assume for all elements~\(j\) that the spectrum~$\{\beta_{0}=0, -\beta_{j1}, -\beta_{j2}, \ldots\}$ of the dissipative operator~$\cL_j$ is discrete with negative real parts bounded away from zero (apart from the single zero eigenvalue): $0>-\beta\geq\Re(-\beta_{j1})>\Re(-\beta_{j2})>\cdots$ for some real, positive decay rate~$\beta$.
\end{assumption} 
For example, the linearised problem~\eqref{eq:spdel} for \bspde\ is spatial diffusion which within each element has a spectrum which is the same for all elements and given by the negative of the decay rates $\beta_k={\pi^2k^2}/{H^2}$\,, for integer $k=0,1,2,\ldots$\,; the corresponding (generalised) eigenfunctions in each element are the modes%
\footnote{The corresponding adjoint (generalised) eigenfunctions are $(1-|\theta|/\pi)\cos k\theta$\,, $\sin k\theta$ and $\sin k|\theta|$ because the adjoint boundary conditions to those in~\eqref{eq:spdel} are $u'_j=0$ at $x=\xpm$\,, $u'_j$~is continuous at~$\xj$, and the nonlocal derivative condition $u'_{jx}(\xj^+,t)-u'_{jx}(\xj^-,t)=u'_{jx}(X_{j+1},t)-u'_{jx}(X_{j-1},t)$.}
\begin{equation}
\cos k\theta\text{ \ for even }k,\quad
\sin k\theta\text{ \ for }k\geq1\,,\quad
\theta\sin k\theta\text{ \ for even }k\geq2\,,
\label{eq:eigfns}
\end{equation}
where the subgrid variable $\theta=\pi(x-\xj)/H$~measures subgrid position relative to the centre grid point within each element (the $j$th~element lies between $\theta=\pm\pi$).
The $k=0$ mode, corresponding to a constant eigenfunction in each element, is linearly neutral as its decay rate~$\beta_0=0$\,.
Thus, linearised about each equilibria in the subspace~$\mathbb E_0$, subgrid structures within each element decay so that a global piecewise constant field emerges exponentially quickly---at least as fast as~$\exp({-\beta' t})$ for any~$\beta'<\beta\leq\pi^2/H^2$ which separates the emergent slow mode~$\beta_0$ from the decaying modes.

\ncm\ theory asserts that such emergence is robust to nonlinear and nonautonomous perturbations.
To address nonautonomous effects we need to expand them in a complete set of basis functions.
Because we embed the dynamics in overlapping elements with `overlapping' eigenfunctions, such as~\eqref{eq:eigfns}, choose a subset~$\eigf_k(x-\xj)$ of the eigenfunctions that are complete over the non-overlapping domains $|x-\xj|<H/2$; for \bspde\ decompose the nonautonomous effects within each element as a linear combination of 
\begin{equation}
    \eigf_k(x-\xj)= \cosin k\theta=\left\{
    \begin{array}{ll}
        \cos k\theta  &\text{for even }k,\\
        \sin k\theta  &\text{for odd }k.
    \end{array}
    \right.
    \label{eq:four}
\end{equation}
This decomposition is analogous to Example~5.2.2 of Da~Prato \& Zabczyk~\cite[see also p.259]{DaPrato96b}.
Thus additive nonautonomous forcing terms in the element \npde~\eqref{eq:spdei} for the \bspde\ are
\begin{equation}
    \phi_j(x,t)=\sum_{k=0}^\infty \phi_{j,k}(t)\eigf_k(x-\xj)
    =\sum_{k=0}^\infty \phi_{j,k}(t)\cosin k\theta\,,
    \label{eq:onoise}
\end{equation}
where $\phi_{j,k}$ denotes the nonautonomous dynamics of the $k$th~wavenumber in the $j$th~element. 
Simple numerical methods, such as Galerkin projection onto the coarsest mode~$\eigf_0$, would ignore the subgrid modes~$\eigf_k$, $k\geq1$\,, of the nonautonomous forcing~\eqref{eq:onoise} and hence miss subtle but potentially important subgrid and inter-element interactions such as those seen in the models \eqref{eq:bssdm} and~\eqref{eq:lowg}.
In contrast, the systematic nature of our application of \ncm\ theory accounts for subgrid microscale interactions as an asymptotic series in the nonautonomous amplitude~$\varepsilon$, inter-element coupling~$\gamma$ and  nonlinearity~$\alpha$.


\subsection{A nonautonomous slow manifold exists} 
\label{sec:exists}
The nonlinear forced \npde~\eqref{eq:spdei} with inter-element coupling conditions~\eqref{eq:ibc} linearises to the dissipative \pde~\eqref{eq:spdel}.
To account for nonzero parameters, adjoin the system of three trivial \textsc{de}s $d\pars/dt =\vec 0$\,, where $\pars=(\varepsilon,\gamma,\alpha)$\,. 
The subspace of equilibria is now~$\{(\vec u,\pars)\mid \vec u\in\mathbb E_0,\ \pars=\vec 0\}$.
In the extended state space~$(\vec u,\pars)=(u_1,\ldots,u_m,\varepsilon,\gamma,\alpha)$, the linearised \pde\ has $m+3$~eigenvalues of zero and all other eigenvalues have negative real part~$\leq-\beta$: for example, this upper bound is $-\pi^2/H^2$ for the \bspde.
Thus the \npde~\eqref{eq:spdei} has an $m+3$~dimensional slow subspace (a slow subspace is characterised by zero eigenvalues instead of the more general centre manifold character of eigenvalues of zero real-part).  
Because of the pattern of the eigenvalues, \ncm\ theory (\cite[Remark~6.2]{Aulbach96} and~\cite[Chapt.~2]{Haragus2011})
assures us a corresponding $m+3$~dimensional nonautonomous slow manifold (\nsm), tangent to the slow subspace, exists under certain conditions.

\begin{assumption} \label{ass:aw}
In addition to the spectral assumption~\ref{ass:sss},
assume \(\cL(U_j)\) is locally integrable, and $\cL$,  $f$~and~$\phi$ are~\(C^p\) in~\(u\) and strongly measurable in~\(t\) \cite[B1 and~B2]{Aulbach2000}.
\end{assumption}

\begin{corollary}[existence]\label{lem:exist}
Under assumption~\ref{ass:aw}, then in some finite neighbourhood~\(N\) of the subspace of equilibria,~$(\mathbb E_0,\vec 0)$, there exists a \(C^p\), $m+3$~dimensional, nonautonomous slow manifold~\(\cM_0\) for the general \npde~\eqref{eq:spdei} \cite[Theorem~2.9]{Haragus2011} in which the field in the $j$th~element is 
\begin{equation}
u_j(x,t)=v_j(\vec U,x,t,\pars)
\quad\text{such that}\quad
\dot \uj =g_j(\vec U,t,\pars),
\label{eq:smx}
\end{equation}
for some function~$g_j$, and where the $j$th component~$\uj(t)$ of vector~$\vec U$ measures the amplitude of the neutral mode~$\eigf_0(x-\xj)$ in the $j$th~element.
Further, \(\cM_0\)~contains the set of bounded solutions of~\eqref{eq:spdei} which stay in~\(N\) for all time \(t\in\RR\).
\begin{itemize}
\item If additionally $\cL$,  $f$ and~$\phi$ are Lipschitz, with small enough Lipschitz constant, and bounded \cite[B2]{Aulbach2000}, then the slow manifold exists globally in state space.
\item For the case of stochastic~\(\phi\) additional conditions are currently required for the existence of a \nsm{}~\cite[Theorem~3, p.212]{Roberts05c}:  
either the nonlinear \npde~\eqref{eq:spdei} is effectively finite dimensional\footnote{A \npde\ is effectively finite dimensional if there exists a wavenumber~$K$ such that modes~$\eigf_k$ for $k\geq K$ do not affect, through~$\cL$, $f$~or~$\phi$, the dynamics for modes with $0\leq k<K$\,.}~\cite[Theorems 5.1 and~6.1]{Boxler89}; or
the nonautonomous forcing in the \npde~\eqref{eq:spdei} is multiplicatively linear in~$u$, $\phi=u\psi(x,t)$ \cite[Theorem~A]{Wang06}.
\end{itemize}
\end{corollary}

\paragraph{Coarse grain nonautonomous lattice dynamics}
Instead of a field continuous in space~$x$, suppose the microscale quantities, including the nonautonomous effects,  are known on a microscale lattice indexed by integer~$i$ as in the spatially discrete system~\eqref{eq:bssde}: for definiteness say $x_i=iH/2$  so that the coarse grid~$X_j$ has twice the spacing of the fine grid~$x_i$ and $X_j=x_{2j}$.  
Then the microscale field is not the $L$-periodic continuum field~$u(x,t)$ (where $L=mH$) but the $2m$-periodic discrete~$u_i(t)$ with $u_{i+2m}(t)=u_i(t)$.
Following a previous linear, autonomous exploration~\cite{Roberts08c}, consider the nonautonomous lattice dynamics~\npdee\ when the dissipative operator is the specific second central difference $\cL u_i(t)=(4/H^2)\left[u_{i+1}(t)-2u_i(t)+u_{i-1}(t)\right]$ with some general nonlinearity~$f$ and nonautonomous effects~$\phi$.
Embed the lattice dynamics onto $m$~overlapping elements to form the system~\eqref{eq:spdei} at internal lattice points and with the same coupling conditions~\eqref{eq:ibc}: the $j$th~element consists of microscale lattice points $x_i\in\{X_j, X_j\pm\frac12H,X_j\pm H\}$ since $i=2j,2j\pm 1,2j\pm2$\,.  
Such a discrete system has a set of piecewise constant equilibria~$(\vecs E_0,\vec 0)$, and about each of the $\vec u$ equilibria the linearised dynamics produce the same dissipative equation~\eqref{eq:spdel} as when~$x$ is continuous.
The spectrum on each element, due to the second central difference~$\cL_j$, is then simply $\{0,-\beta_1,-\beta_2\}=\{0,-8/H^2,-16/H^2\}$ corresponding to microscale lattice eigenmodes (wavelets perhaps) on each element of $\eigf_0=(1,1,1,1,1)$, $\eigf_1=(0,-1,0,1,0)$ and $\eigf_2=(1,-1,1,-1,1)$, respectively~\cite{Roberts08c}.
Consequently, the Existence Corollary~\ref{lem:exist} then guarantees that for certain classes of lattice nonlinearity~$f$ and nonautonomous effects~$\phi$ there exists a \nsm\ describing an in principle exact closure, such as the approximate~\eqref{eq:bssdm}, of the microscale nonautonomous dynamics on the coarse macroscale grid values.

It will not escape your notice that such sound mapping of nonautonomous dynamics from a fine grid to a grid a factor of two coarser, such as \eqref{eq:bssde}$\mapsto$\eqref{eq:bssdm}, may be iterated across all grid scales on an entire multigrid hierarchy.
This approach has the potential to explore nonautonomous microscale dynamics at any scale, and to strongly relate dynamics across any scales.
Thus this approach contributes to the need identified by Dolbow et al.~\cite[p.4]{Dolbow04} for ``representing information transfer across levels of scale'', and the more recent call by Brown et al.~\cite[p.14]{Brown08} for ``adaptive multiscale discrete stochastic simulation methods that are justified by theory and which can automatically partition the system into components at different scales'', as our quite general treatment of the nonautonomous effects~$\phi$ permits a stochastic form.

\subsection{The nonautonomous slow manifold captures emergent dynamics}
The second key property of \nsm{}s is that the evolution on the \nsm\ does capture the long term dynamics of the original \npdee, apart from exponentially decaying transients.
For example, all solutions of \bspde{} close enough to the origin are are exponentially quickly described by the discrete \npde{}s~\eqref{eq:lowg}.
This strong theoretical support for the model holds at finite element size~$H$---it ensures an accurate closure for the macroscale discretisation.

For nonautonomous systems, such emergence of the \nsm\ is most clearly seen via time dependent, normal form, coordinate transforms~\cite[e.g.]{Arnold98, Roberts06k, Roberts07d}.
Such a normal form coordinate transform underlies the topological equivalence that establishes the following corollary on emergence.


\begin{corollary}[emergence] \label{lem:rel} 
For the conditions of Corollary~\ref{lem:exist}, the finite neighbourhood~$N$ of the equilibria~$(\mathbb E_0,\vec 0)$ may be chosen such that while each solution~$u_j(x,t)$ of the \npde~\eqref{eq:spdei} remains in the neighbourhood~$N$,  there exists a solution~$\vec U(t)$ on the \nsm~\(\cM_0\) such that $\|u_j(x,t)-v_j(\vec U,x,t,\pars)\|=\Ord{e^{-\beta't}}$ as as time~$t$ increases.
\end{corollary}
\begin{proof}
In the case of Lipschitz and bounded~\(\cL\), \(\vec f\) and~\(\phi\), Aulbach and Wanner~\cite[Theorem~4.1]{Aulbach2000} proved the topological equivalence between the nonlinear dynamics of the stable modes and the linearised stable modes, hence the slow manifold emerges globally.
For the more general case, invoke a cut-off function (\cite[Remark~6.2]{Aulbach96}, \cite[\S B.1]{Haragus2011}) and immediately establish the local existence of a finite neighbourhood~\(N\) in which the topological equivalence occurs.
Inside~\(N\) the stable modes are topologically equivalent to the decaying solutions of the linearised system~\eqref{eq:spdel}.
Consequently, there exists a decay rate~\(\beta'\leq\beta\) of attraction to a solution on~\(\cM_0\).  
Boxler also assures us that for stochastic forcing the rate of decay to the \nsm\ is comparable to the deterministic case~\cite[Theorem~7.1(i)]{Boxler89}.
\end{proof}

For example, in \bspde\ on times significantly larger than a cross element diffusion time~$\beta_1^{-1}=H^2/\pi^2$, the exponential transients decay and the \nsm\ model~\eqref{eq:lowg} describes the dynamics.
Similarly, the transients of the discrete microscale dynamics~\eqref{eq:bssde} decay on a time scale of~$\beta_1^{-1}=H^2/8$ to the \nsm\ model~\eqref{eq:bssdm}.

Subtleties in this Emergence Corollary mislead some researchers, even when modelling deterministic systems.
For example, Givon et al.~\cite{Givon04} discuss finite dimensional deterministic systems which linearly separate into slow modes~$x(t)$ and fast, stable modes~$y(t)$. 
They identify the existence of a slow exponentially attractive invariant manifold $y=\eta(x)$ and approximate the effective low~dimensional evolution on the manifold as $\dot X=L_1X+f(X,\eta(X))$\,.
However, they \emph{assume}~\cite[p.R67, bottom]{Givon04} that the initial condition for the evolution on the slow manifold is simply $X(0)=x(0)$, and then consequently, just after their~(4.5), have to place undue restrictions on the possible initial conditions.
However, the source of such restriction is that in general $X(0)\neq x(0)$ since the initial condition is not required to lie on the slow manifold (it simply needs to be in the neighbourhood of the manifold). 
Physicists sometimes call the difference between $x(0)$~and~$X(0)$ the `initial slip'~\cite[e.g.]{Grad63}.
For nonautonomous systems, the correct nontrivial projection of initial conditions onto the \nsm\ may be constructed via nonautonomous normal forms~\cite[e.g.]{Arnold98, Roberts06k, Roberts07d} and applies in some finite domain around the subspace~$(\vecs E_0,\vec 0)$.

There are two caveats to our application of Corollary~\ref{lem:rel} to discretising \npde{}s.
Firstly, although our constructed asymptotic series are global in the $m$~coarse macroscale amplitudes~$\vec U$, they are local in the parameters $\pars=(\varepsilon,\gamma,\alpha)$: the rigorous theoretical support applies in some finite neighbourhood of $\pars=\vec 0$\,.
At this stage we have little information on the size of that neighbourhood.
In particular, we evaluate the model when coupling parameter $\gamma=1$ to recover a discrete model for fully coupled elements; thus we require $\gamma=1$ to be in the finite neighbourhood of validity.
For example, the truncation in powers of the coupling parameter~$\gamma$ controls the width of the computational stencil for the discrete models.
Due to the form of the coupling conditions~\eqref{eq:ibc}, nearest neighbour elements interactions are flagged by terms in~$\gamma^1$, whereas interactions with next to nearest neighbouring elements occur as $\gamma^2$~terms, and so on for higher powers.
The low accuracy models \eqref{eq:bssdm} and~\eqref{eq:lowg} are constructed with error~$\Ord{\gamma^2}$ and so only encapsulate interactions between the dynamics in an element and those of its two adjoining neighbours.
Construction to higher orders in coupling~$\gamma$ accounts for interactions between more neighbouring elements.
Secondly, we cannot construct the \nsm\ and the evolution thereon exactly; it is difficult enough constructing asymptotic approximations such as the low order accuracy models \eqref{eq:bssdm} and~\eqref{eq:lowg}.
The models we develop and discuss have an error due to the finite truncation of the asymptotic approximations in the small parameters~$\pars$.
Nonetheless, Haragus and Iooss~\cite[Corollary~2.12]{Haragus2011} establish that our description~\eqref{eq:smx} of the slow manifold should simply satisfy the \npde~\eqref{eq:spdei} and coupling conditions~\eqref{eq:ibc} via straightforward application of the chain rule.

\section{Nonlinear dynamics have irreducible microscale interactions}
\label{sec:quad}
Using the specific example of the forced \bspde, we now explore typical issues arising in constructing macroscale discretisations of the quite general nonlinear  \npdee, issues that arise for both differential and difference \npde{}s.

\subsection{Separate products of convolutions}
\label{sec:spc}

As detailed elsewhere~\cite[\S3.1]{Roberts06g}, iteration is a powerful and flexible method to construct the \nsm. 
The aim is to construct the functions~\(v_j\) and~\(g_j\) of the \nsm~\eqref{eq:smx} that satisfy, to some order of error, the governing equations \cite[Corollary~2.12]{Haragus2011}.
Suppose that at some iterate we know $u_j(x,t)\approx v_j(\vec U,x,t,\pars)$ such that $d\uj/dt\approx g_j(\vec U,\pars)$ correct to some order of error. 
We use the residuals of the governing \npde~\eqref{eq:spdei} and coupling conditions~\eqref{eq:ibc} to derive corrections of~$v_j$ and~$g_j$; that is, we seek corrections~$v_j'$ and~$g_j'$ where \(u_j\approx v_j+v'_j\) and \(d{U_j}/dt\approx g_j+g'_j\) is a better approximation to the \nsm.
To make a better approximation, the corrections must satisfy~\cite[\S3.1]{Roberts06g}
\begin{equation}
    \D t{v'_j}-\cL_j{v'_j}+g'_j
    =\text{residual}_{\eqref{eq:spde}}\,.
    \label{eq:diffd}
\end{equation}
In analysing a nonlinear \npdee, such as the  forced \bspde, products of memory convolutions appear in the residual.
It seems appropriate, for reasons discussed further on, to seek terms at most quadratic in the nonautonomous coefficient~$\varepsilon$.
Thus we generally have to deal with quadratic products of multiple convolutions in time.

\paragraph{Obtain corrections from residuals}

To cater for the general case, define multiple convolutions.
First, let the operator~$\Z{\kappa}$ denote convolution over past history with $\exp[-\beta_{\kappa} t]$ where for brevity~$\kappa$ denotes the pair~$jk$ corresponding to the $k$th~eigenvalue of the $j$th~element.
That is,
\begin{equation}
        \Z{\kappa}\phi_{\mu}(t)
	=\exp[-\beta_{\kappa}t]\star\phi_{\mu}(t) =\inti^t \exp[-\beta_{\kappa}(t-\tau)]
	\phi_{\mu}(\tau) \,d\tau\,,
    \label{eq:conv}
\end{equation}
where $\mu$, like $\kappa$, represents an eigenvalue and element  pair; 
recall that $\beta_{\kappa}$~is the (positive) decay rate of the $k$th~mode within the $j$th~element; $\beta_{jk}=k^2\pi^2/H^2$ for the \bspde, independent of the element.
Second, let $\Z{\vec {\kappa}}$ denote the operator of multiple convolutions in time where components of the subscript vector~$\vec {\kappa}$ indicate the decay rates of the multiple convolutions, that is, the operator
\begin{equation}
    \Z{\vec {\kappa}}=\Z{({\kappa}_1,{\kappa}_2,\ldots)} =\exp(-\beta_{{\kappa}_1}t)\star
    \exp(-\beta_{{\kappa}_2}t)\star \cdots\star
    \quad\text{and}\quad
    \Z{\kappa}=1\,,
    \label{eq:Z}
\end{equation}
in terms of the convolution~\eqref{eq:conv}; consequently
\begin{equation}
    \partial_t \Z{({\kappa}_1,{\kappa}_2,\ldots)}
    =-\beta_{{\kappa}_1} \Z{({\kappa}_1,{\kappa}_2,\ldots)}
    +\Z{({\kappa}_2,\ldots)}\,.
    \label{eq:Zt}
\end{equation}
The order of the convolutions does not matter\footnote{
This can be seen by writing $\phi_{\mu}(t)$ in terms of its Fourier transform and then evaluating all time integrals in the convolution; the solution is symmetric in all $\beta$.}~\cite{Roberts05c}; 
however, keeping intact the order of the convolutions seems useful to most easily cancel like terms in the residual of the governing \npde.

Earlier work~\cite[\S3.3--4]{Roberts06g} explored updates $v'_j$~and~$g'_j$ to the subgrid slow manifold field and the discrete model \npde\, such as  \eqref{eq:bssde}~\cite[e.g.]{Roberts08c} with linear diffusion, although only autonomous~$\phi$ was considered in the latter.
Summarising, in iteratively constructing the \nsm\ we encounter convolution integrals over the immediate past of the nonautonomous effects~$\phi$.
Such fast time `memory' convolutions must be removed from the dynamics of the discretisation $\dot{U}_j=g_j$\,: Givon et al.~\cite[p.R59]{Givon04} similarly comment ``Memory.
An important aim of any such algorithm is to choose~$P$ [the \nsm] in such a way that the dynamics in~$X$ [our grid values~$\vec U$] is memoryless.'' 
We simplify the discrete model tremendously by removing such `memory' convolutions as originally developed for \sde{}s by Coullet, Elphick \& Tirapegui~\cite{Coullet85}, Sri Namachchivaya \& Lin~\cite{Srinamachchivaya91}, and Roberts \& Chao~\cite{Chao95, Roberts03c}: the trick is to absorb the `memory' convolutions into the parametrisation of the \nsm, via~$v_j$, leaving the evolution on the \nsm,~$g_j$, to be convolution free.
The results are discrete models, such as~\eqref{eq:bssdm}, where the nonautonomous effects are correlated across neighbouring elements despite there being no correlations in the original linear diffusion \bspde~and~\eqref{eq:bssde}.

Nonlinear spatially extended problems, such as the \bspde, need to adapt extra techniques from \npde{}s on small spatial domains~\cite{Roberts05c}.
For nonlinear problems we additionally have to solve for nonautonomous corrections that are quadratic in~\(\phi\).
Generally the right-hand side of the correction equation~\eqref{eq:diffd} contains a sum of terms of the form
\begin{equation}
	\D t{v'_j}-\cL_j{v'_j}+g'_j = \varepsilon^2F_j(x)\Z{\vec{\lambda}}\phi_{\rho}\Z{\vec {\kappa}}\phi_{\mu}\,,
	\label{eq:wupnon}
\end{equation}
where $F_j(x)$ denotes complicated expressions encapsulating some of the influences of surrounding elements upon the subgrid structures within the $j$th~element.
Analogous to the treatment of linear nonautonomous effects~\cite[\S3.3--4]{Roberts06g}, two cases arise.
\begin{itemize}
    \item  Firstly, for each component of the subgrid
    structure~$F_j(x)$ in $\eigf_p(x-\xj)$ for wavenumber $p\geq
    1$\,, there is no difficulty in simply adding into the
    correction~$v'_j$ to the subgrid field the component
    \begin{displaymath}
        \varepsilon^2\eigf_p(x-\xj) \Z{jp}\left[ \Z{\vec{\lambda}}\phi_{\rho}\Z{\vec{\kappa}}\phi_{\mu} \right]
    \end{displaymath}
    with its extra convolution in time.

	\item Secondly, for the component in~$F_j(x)$ that is constant across an element, the $\eigf_0(x-\xj)$~component, we separate the part of $\Z{\vec \lambda}\phi_{\rho}\Z{\vec {\kappa}}\phi_{\mu}$ that has a bounded 	integral in time, and hence updates the subgrid field~$v'_j$, from the secular part that does not have a bounded integral, and hence must update the model \npde\ through~$g'_j$.
\end{itemize}
In the considerations and convolutions for either case, the surrounding grid values that appear in the spatial structure forcing~$F_j$ are treated as constants as the time derivative in~\eqref{eq:wupnon} is the partial derivative of~$v_j(\vec U,x,t,\pars)$ keeping the grid values~$\vec U$ constant.
The $g'_j$~term in~\eqref{eq:wupnon} accounts for the time derivatives of grid values~$\vec U$.

\paragraph{Integrate by parts to separate}

Integration by parts introduced in earlier research~\cite{Roberts03c, Roberts05c} also here reduces all non-integrable convolutions to the quadratic canonical form of the convolution.
Summarising, let's choose the canonical irreducible form for quadratic interactions to be $\phi_{\rho}\Z{\vec {\kappa}}\phi_{\mu}$\,.  
We rewrite the convolution \ode~\eqref{eq:Zt} as $\beta_{\kappa}\Z{{\kappa}.\vec {\kappa}'}=-\partial_t{\Z{{\kappa}.\vec {\kappa}'}}+\Z{\vec {\kappa}'}$\,, where the vector of convolution parameters is decomposed as $\vec {\kappa}={\kappa}\cdot\vec {\kappa}'$ so that ${\kappa}$~is the first component of vector~$\vec {\kappa}$, and $\vec {\kappa}'$~is the vector (if any) of the second and subsequent components of vector~$\vec {\kappa}$.
Then for any $\rho$ and~$\mu$, whether from the same element or not,
\begin{eqnarray*}
    \int \Z{\lambda\cdot\vec \lambda'}\phi_{\rho}  \Z{{\kappa}\cdot\vec {\kappa}'}\phi_{\mu}\,dt
    &=&-\frac1{\beta_{\kappa}+\beta_{\lambda}}\Z{\lambda\cdot\vec
    \lambda'} \phi_{\rho} \Z{{\kappa}\cdot\vec {\kappa}'}\phi_{\mu}
    \\&&{}
    +\frac1{\beta_{\kappa}+\beta_\lambda}\int \left(\Z{\vec {\lambda}'}\phi_{\rho}
    \Z{\kappa\cdot\vec \kappa'}\phi_{\mu} \right.
    \\&&{}
\left.   \quad\quad\quad\quad +\Z{{\lambda}\cdot\vec {\lambda}'}\phi_{\rho} \Z{{\vec \kappa}'}\phi_{\mu}\right) dt\,.
\end{eqnarray*}
Observe that each of the two components in the integrand on the right-hand side has one fewer convolutions than the initial integrand.
Thus one repeats this integration by parts until terms of the canonical form $\phi_{\rho}\Z{\vec {\kappa}}\phi_{\mu}$ in the integrand are reached.
In this reduction process, assign all the integrated terms to update the subgrid field~$v'_j$. 
The irreducible terms remaining in the integrand, those in the form $\phi_{\rho}\Z{\vec {\kappa}}\phi_{\mu}$, must thus go to update the nonautonomous evolution~$g'_j$.

Computer algebra~\cite[\S6]{Roberts06b} readily implements these steps in the iteration to derive the asymptotic series of the \nsm\ of an \npdee.

\subsection{Odd modes highlight nonlinear interactions}
\label{sec:onhnni}

Subgrid modes with even wavenumber~$k$ are the only modes to affect the  discrete macroscale model of linear diffusion~\cite[\S3.3]{Roberts06g}.
For \emph{nonlinear} dynamics, any term in the discrete macroscale model dependent on an odd subgrid mode must be a nonlinear effect.
Thus, restricting the nonautonomous effects~$\phi(x,t)$ to an odd subgrid structure within each element highlights the spatially extended dynamics of the \bspde\ arising from nonlinear dynamics.
These results clarify the modelling of subgrid nonlinear fluctuation effects in the general \npdee.

\begin{figure}
    \centering
\includegraphics{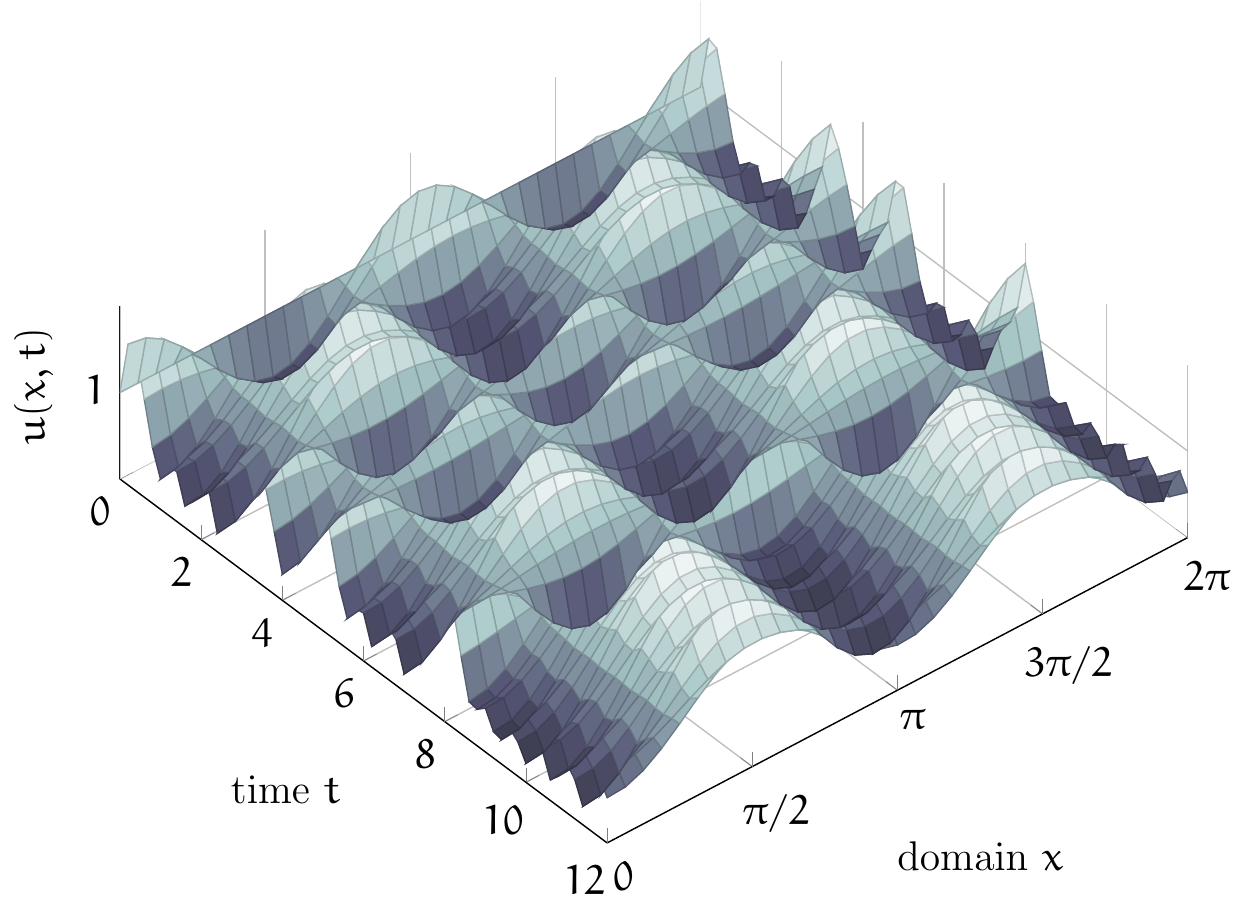}
	\caption{microscale simulation of one realisation of the $2\pi$-periodic \bspde\ with initial condition $u(x,0)=1$\,,  nonlinearity $\alpha=0.3$ and $\varepsilon=0.05$\,. 
The nonautonomous forcing is $\varepsilon\varphi(t)\cos 2x$ for $\varphi=\xi$ obtained from the Lorenz system. 
The Lorenz system is as in Figure~\ref{fig:micro}, with the exception that here there is only one Lorenz component across the fine space mesh and $\xi(0)=10$\,.
The macroscale grid points \(X_j=(2j-1)\pi/4\) are at the nodes of the forcing where the field~$u$ is relatively quiescent.}
    \label{fig:firsts}
\end{figure}

\paragraph{One correlated noise}
The simplest nontrivial case of \bspde\ is when the nonautonomous forcing~$\phi(x,t)$ has just the one odd Fourier component~$\sin\theta$ in each element~$j$, \emph{and} is perfectly but oppositely correlated in neighbouring elements~$j\pm 1$.
That is, in this section set the nonautomous dynamics associated with the first (and only) Fourier component of 
$\phi$ to $\phi_{j,1}=(-1)^j\varphi(t)$ in each element~$j$ so that 
\begin{equation}
    \phi_j(x,t)=(-1)^j\varphi(t)\sin\theta\,,
    \label{eq:first}
\end{equation}
for some smooth time dependent forcing~$\varphi(t)$. 
For example, Figure~\ref{fig:firsts} shows a microscale simulation of $2\pi$-periodic \bspde: as the macroscale element size $H=\pi/2$\,, the space-time structure of the nonautonomous forcing~\eqref{eq:first} reduces to $\phi(x,t)=\varphi(t)\cos2x$ across all elements~$j$. 
We proceed to model such microscale nonautonomous dynamics with four elements, with $H=\pi/2$\,, centred at $\xj=(j-\rat12)\pi/2$\,.
The nodes of the nonautonomous forcing~\eqref{eq:first} are at the grid points~$\xj$, that is, $\phi_j(\xj,t)=0$ for all $j$. 
Furthermore, the averages of the nonautonomous forcing across each of the four elements centred about the macroscale grid points~$\xj$ are zero.  
Most methods for discretising \bspde\ on these elements would sample or average over this nonautonomous forcing and predict it has no influence whatsoever on the macroscale grid evolution.
However, the subgrid scale nonlinear advection in \bspde\ carries the subgrid scale forcing past the grid points and so generates fluctuations along the grid values~$\xj$: these nonautonomous fluctuations in $u(\xj,t)$ are relatively small and hard to see in the simulation of Figure~\ref{fig:firsts}, but are clear in Figure~\ref{fig:firstg2s1ah} (red curve).
In contrast to most methods for discretising \npde{}s, our holistic discretisation provides a systematic closure for the subgrid microscale dynamics and so correctly predicts the spatially extended dynamics  which produce fluctuations in the macroscale grid values~$\xj$.  

\begin{figure}
    \centering
\includegraphics[width=12cm]{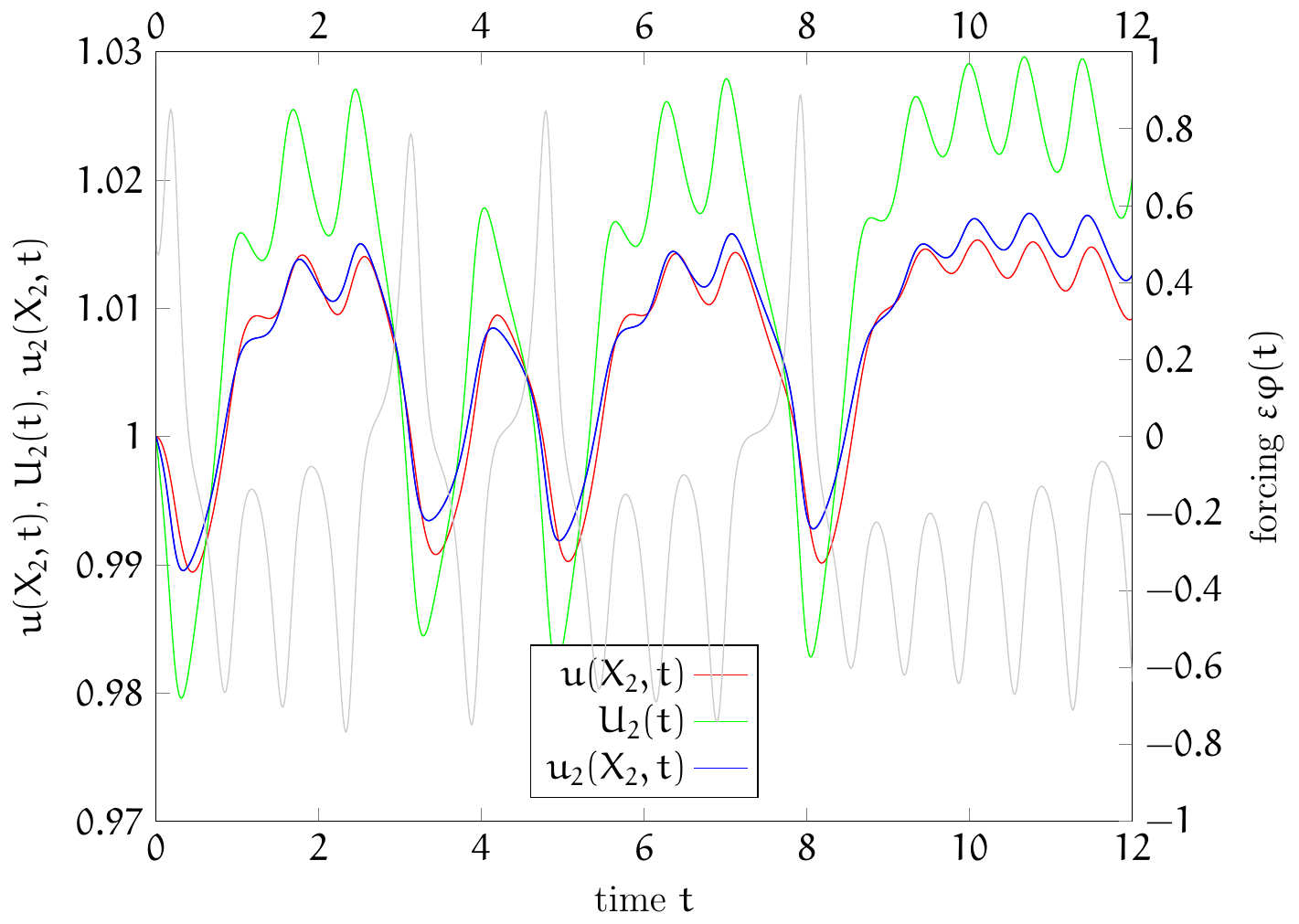}
	\caption{Compare the macroscale model~\eqref{eq:ssm1dt} with a microscale simulation for  nonlinearity $\alpha=0.3$, nonautonomous  amplitude $\varepsilon=0.05$ and for $\varphi$~defined by the Lorenz system (as in Figure~\ref{fig:firsts}): 	red, the microscale field $u(X_2,t)$ showing the subgrid forcing carried by nonlinear advection past the grid point; green, the macroscale variable~$U_2(t)$; and blue, the macroscale \nsm~\eqref{eq:ssm1g} at the grid point,~$u_2(X_2,t)=v_2(\vec U,X_2,t,\pars)$. 
The \nsm\ and the microscale simulation match well. 
The fluctuations in the fields follow those of the nonautonomous effects at $x=0$, $\varepsilon\varphi(t)$ (light gray), 
despite the effects being strictly zero along~$X_2$, $\phi(X_2,t)=0$\,.}
    \label{fig:firstg2s1ah}
\end{figure}

Computer algebra~\cite{Roberts06b} constructs the \nsm\ of \bspde\ with interelement coupling~\eqref{eq:ibc}: just modify the code to the desired forcing~\eqref{eq:first}.
In each element, with nonautonomous effects induced and subgrid structures truncated to the first eight subgrid Fourier components~$1$, $\cosin\theta$, \ldots, $\cosin7\theta$\,, the subgrid microscale field is
\begin{subequations} \label{eq:ssm1m}%
\begin{eqnarray}&&
    u_j(x,t)=\uj
    +\gamma\Big[\frac{\theta}{\pi}\mud
    +\frac{\theta^2}{2\pi^2}\delta^2 \Big]\uj
    \pm\varepsilon\sin\theta\,\Z1\varphi
        \label{eq:ssm1ma}
\\&&{}
    \pm\varepsilon\alpha\uj\Big[
    \frac{2H}{\pi^2}\Z1 \label{eq:ssm1mb}
\\&&{}
    -\frac4H\big( \rat13\cos2\theta\,\Z{2,1} -\rat1{15}\cos4\theta\,\Z{4,1}
    +\rat1{35}\cos6\theta\,\Z{6,1} \big)\Big]\varphi
        \ \qquad\label{eq:ssm1mc}
\\&&{}
    +\Ord{\varepsilon^3+\alpha^3+\gamma^{3/2}}\,,
    \label{eq:ssm1md}
\end{eqnarray}%
\end{subequations}
where the upper alternative is for even~$j$, and the lower alternative for odd~$j$.  
The terms~\eqref{eq:ssm1mb}--\eqref{eq:ssm1mc} in the \nsm\ begin to account for the nonlinear advection and the interactions of subgrid spatial structures: these processes transform the additive forcing into multiplicative forcing,~$\uj\varphi$, with memory via the convolutions~$\Z{p,1}$. 
For conciseness, the subscripts of these and subsequent convolutions in this section only indicate eigenvalue modes and not the element number~$j$ since, for the decay rates of Burgers' \npde{}, $j$ is superfluous; $\beta_{jk}=\beta_k$ for all~$j$. 
Higher order terms in the coupling~$\gamma$, nonlinearity~$\alpha$ and forcing magnitude~$\varepsilon$ are too onerous to record for the microscale subgrid structures.

Potzsche and Rasmussen~\cite[Proposition~3.6]{Potzsche2006} justify the asymptotic error reported in the \nsm~\eqref{eq:ssm1m}.
This reported error comes from the termination criterion of the iterative construction of the \nsm.
Computer algebra~\cite{Roberts06b} iterates until the residual of the governing \bspde\ (or the fine lattice dynamics~\eqref{eq:bssde}), and the residual of the interelement coupling conditions~\eqref{eq:ibc}, are of some specified asymptotic order of error.
Potzsche and Rasmussen's~\cite{Potzsche2006} Proposition~3.6
(or in the case of stochastic dynamics, Boxler's~\cite{Boxler89} Theorem~8.1)
then guarantees that the \nsm\ model constructed by the computer algebra has the same asymptotic order of error, as reported in the \nsm~\eqref{eq:ssm1m}.  

One may straightforwardly check by hand the error of an \nsm, such as~\eqref{eq:ssm1m}, by direct substitution into \bspde\ and the coupling conditions~\eqref{eq:ibc}.
However, the algebraic details are enormous, reflecting the intricate subgrid scale interactions and including the transformations outlined in Section~\ref{sec:spc}, and would fill many pages with otiose algebraic expressions.
Table~\ref{tbl:newt} tabulates the number of terms of various orders for the evolution on the \nsm\ of the \bspde; the \ode{}s~\eqref{eq:ssm1dt} are one example of this family. 
The number of terms describing the shape of the \nsm---higher order versions of~\eqref{eq:ssm1m}---is several orders of magnitude more.
Surely it is far better to leave such intricate detail to a computer, and focus instead on the nonautonomous model, its theoretical support, and its implications as we do here.

The construction of the \nsm~\eqref{eq:ssm1m} proceeds `hand-in-hand' with constructing the evolution on the \nsm.
Iterating to higher order in the small parameters~$\pars$, computer algebra~\cite{Roberts06b} gives that with the simple nonautonomous effects~\eqref{eq:first}, the emergent evolution of the grid values~$\uj(t)$ on the \nsm~\eqref{eq:ssm1m} is the system of \npde{}s
\begin{subequations} \label{eq:ssm1dt}%
\begin{align}
    \duj ={}& \frac\gamma{H^2}\delta^2\uj
    -\frac{\gamma^2}{12H^2}\delta^4\uj
    -\alpha\frac{\gamma}{H}\uj\mud\uj
    +\alpha^2\frac\gamma{12}\uj^2\delta^2\uj
        \label{eq:ssm1dta}
\\&{}
    \mp\varepsilon\alpha H\Big[ 
    \frac{2}{\pi^2} \uj
    +\gamma\big( .1028\, \uj +.0716\, \delta^2\uj \big) 
    -.00363\,\alpha^2H^2\uj^3 \Big]\varphi
        \qquad\label{eq:ssm1dtb}
\\&{}
    +\varepsilon^2\alpha^2\Big[ -\frac{8}{\pi^2}\uj\varphi \big(
    \rat1{15}\Z{2,1} +\rat1{255}\Z{4,1} +\rat1{1295}\Z{6,1} \big)\varphi
        \label{eq:ssm1dtc}
\\&\quad{}
    + 0.0195\, {H^2}\uj\varphi \Z1\varphi \Big]
    +\Ord{\varepsilon^5+\alpha^5+\gamma^{5/2}}.
    \label{eq:ssm1dtd}
\end{align}%
\end{subequations}
The terms~\eqref{eq:ssm1dta} form the deterministic holistic discretisation of Burgers' equation~\cite{Roberts98a}.
The terms~\eqref{eq:ssm1dtb} describe effects linear in the forcing: as for linear diffusion~\cite[\S3.2--3]{Roberts06g}, we remove all memory convolutions from these terms.
However, the terms~\eqref{eq:ssm1dtc}--\eqref{eq:ssm1dtd} quadratic in the forcing,~$\varepsilon^2$, generally must contain memory convolutions (as discussed in Section~\ref{sec:spc}) in order to maintain the system of \npde{}s~\eqref{eq:ssm1dt} as a strong model of \bspde.

Figure~\ref{fig:firstg2s1ah} shows $U_2(t)$ (green curve) for one simulation of the discrete model \npde~\eqref{eq:ssm1dt} for forcing $|\varepsilon \varphi|<1$ (gray curve).
Compare these to the red curve of the microscale field~$u(X_2,t)$ at the corresponding grid point.
Although the overall trends are roughly similar, the macroscale variable~$U_2(t)$ is markedly different to~$u(X_2,t)$. 
How then can the \npde~\eqref{eq:ssm1dt} be a strong model?  
Resolve the difference by recalling that to eliminate memory convolutions we must abandon the freedom to impose precisely the meaning of the amplitudes~$\uj$ \cite[\S3, e.g.]{Roberts05c}.
Thus, generally, $U_2(t)\neq u(X_2,t)$\,.
Instead, the field at a grid point,~$u(\xj,t)$, is predicted by the \nsm~\eqref{eq:ssm1m} evaluated at the grid points, namely
\begin{eqnarray}
    u_j(\xj,t)&=&\uj
    \pm\varepsilon\alpha\uj\Big[ \frac{2H}{\pi^2}\Z1 -\frac4H\big(
    \rat13\Z{2,1} -\rat1{15}\Z{4,1} +\rat1{35}\Z{6,1}
    \big)\Big]\varphi
    \nonumber\\&&{}
    +\Ord{\varepsilon^4+\alpha^4+\gamma^2}\,.
    \label{eq:ssm1g}
\end{eqnarray}
Figure~\ref{fig:firstg2s1ah} plots the \nsm\ predicted grid value~\eqref{eq:ssm1g} in blue and displays good agreement with the microscale simulation (red).
Evidently, \ncm\ theory successfully supports discrete macroscale models of nonlinear \npde{}s.

\subsection{Strong models of nonautonomous dynamics are very complicated}

Now restore independent and multiple forcing processes in each element and consider the details of discretisations of the forced \bspde.
Computer algebra~\cite{Roberts06b} derives the following leading terms in the asymptotic series of the model $d\uj/dt=g_j(\vec U,t,\pars)$. The large amount of algebraic detail reflects the enormous complexity of the multiple physical interactions acting on the subgrid microscale structures controlled by the potentially rich spectrum of nonautonomous fluctuations.
The arguments of the next section simplify this model significantly.

Computer algebra~\cite{Roberts06b} derives that the element amplitudes~$\uj(t)$ evolve according to the system of \npde{}s
\begin{subequations}\label{eq:strongquad}%
\begin{align}
    &\duj=
    \gamma\frac1{H^2}\delta^2\uj 
    -\gamma^2\frac1{12H^2}\delta^4\uj
    -\gamma\alpha\frac1H\uj\mud\uj
    \label{eq:strongquada}\\&{}
    +\varepsilon\left\{ \vphantom{\frac83}
        \left[ 1 -\gamma\rat1{24}\delta^2
        +\gamma^2(\rat3{640}+\rat1{8\pi^4})\delta^4 \right]\phi_{j,0}
    \right.\nonumber\\&\left.\quad{}
        +\left[ \gamma\rat1{4\pi^2}\delta^2 -\gamma^2(\rat1{48\pi^2} 
        +\rat1{16\pi^4})\delta^4 \right]\phi_{j,2}
        -\alpha\frac{2H}{\pi^2}\uj\phi_{j,1}
    \right.\nonumber\\&\left.\quad{}
        +\alpha\gamma\frac{1}{H^2\pi^2}\left[  
            \uj\left( \rat8{\pi^2}\mud\phi_{j,0} 
                -\rat14\mud\phi_{j,2}
                +(\rat1{12}+\rat5{3\pi^2})\delta^2\phi_{j,1}
            \right)
    \right.\right.\nonumber\\&\left.\left.\qquad{}
            +\mud\uj\left( \rat14\phi_{j,2} 
                +(\rat16+\rat{10}{3\pi^2})\mud\phi_{j,1}
            \right)
    \right.\right.\nonumber\\&\left.\left.\qquad{}
            -\delta^2\uj\left( (\rat16+\rat1{3\pi^2})\phi_{j,1} 
                -(\rat1{24}+\rat5{6\pi^2})\delta^2\phi_{j,1} 
            \right)
        \right]
        -\alpha^2\frac{8H^2}{3\pi^4}\uj^2\phi_{j,0}
    \right\}
    \label{eq:strongquadb}
    \\&{}
    +\varepsilon^2\left\{\vphantom{\frac11}
        \alpha\frac{H}{\pi^2}\left[ -2\phi_{j,0}\Z{1}\phi_{j,1}
            +\rat25\phi_{j,1}\Z{2}\phi_{j,2}
            +\rat25\phi_{j,2}\Z{1}\phi_{j,1} 
        \right]
    \right.\nonumber\\&\left.\quad{}
        +\alpha\gamma\frac1{H\pi^2}\left( -32\phi_{j,0}\Z{1,2}\mud
        -\rat45\phi_{j,1}\Z{2,2}\delta^2
        +\rat{32}5\phi_{j,2}\Z{1,2}\mud \right)\phi_{j,2}
    \right.\nonumber\\&\left.\quad{}
        +\alpha\gamma\frac{H}{\pi^2}\left[ 
            \phi_{j,0}\left( \rat8{\pi^2}\Z1\mud(\phi_{j,0}+\phi_{j,2}) 
            +(\rat1{12}+\rat5{3\pi^2})\Z1\delta^2\phi_{j,1}
    \right.\right.\right.\nonumber\\&\left.\left.\left.\quad\qquad{}
            -(\rat14+\rat8{\pi^2})\Z2\mud\phi_{j,2}
            \right)   
            +\phi_{j,1}\Z2\left( \rat15\delta^2\phi_{j,0}
            -(\rat1{20}+\rat{13}{150\pi^2})\phi_{j,2} \right)
    \right.\right.\nonumber\\&\left.\left.\qquad{}
            +\phi_{j,2}\left( -\rat8{5\pi^2}\Z1\mud(\phi_{j,0}+\phi_{j,2}) 
            -(\rat1{60}+\rat{17}{75\pi^2})\Z1\delta^2\phi_{j,1}
    \right.\right.\right.\nonumber\\&\left.\left.\left.\quad\qquad{}
            +(\rat18+\rat4{5\pi^2})\Z2\mud\phi_{j,2}
            \right)
    \right.\right.\nonumber\\&\left.\left.\qquad{}
            +\delta^2\phi_{j,0}\,\Z1\left(
            -(\rat1{12}+\rat2{15\pi^2})
            +(\rat1{24}+\rat5{6\pi^2})\delta^2 \right)\phi_{j,1}
    \right.\right.\nonumber\\&\left.\left.\qquad{}
            -\delta^2\phi_{j,1}\,\Z2\left(
            (\rat1{60}+\rat{17}{75\pi^2})
            +(\rat1{120}+\rat{17}{150\pi^2})\delta^2 \right)\phi_{j,2}
    \right.\right.\nonumber\\&\left.\left.\qquad{}
            -\delta^2\phi_{j,2}\,\Z1\left(
            (\rat1{20}+\rat{44}{75\pi^2})
            +(\rat1{120}+\rat{17}{150\pi^2})\delta^2 \right)\phi_{j,1}
    \right.\right.\nonumber\\&\left.\left.\qquad{}
            +\mud\phi_{j,0}\left(
            (\rat16+\rat{10}{3\pi^2})\Z1\mud\phi_{j,1}
            +(\rat14-\rat8{5\pi^2})\Z2\phi_{j,2} \right)
    \right.\right.\nonumber\\&\left.\left.\qquad{}
           -\mud\phi_{j,1} (\rat1{30}+\rat{34}{75\pi^2}) \Z2\mud\phi_{j,2}
    \right.\right.\nonumber\\&\left.\left.\qquad{}
            +\mud\phi_{j,2}\left(
            -(\rat1{30}+\rat{34}{75\pi^2})\Z1\mud\phi_{j,1}
            +(\rat18-\rat4{5\pi^2})\Z2\phi_{j,2} \right)
        \right]
    \right.\nonumber\\&\left.\quad{}
        +\alpha^2\frac1{\pi^2}\uj\left[ -\rat{16}3\phi_{j,0}\left(
            2\Z{1,2} +\rat{H^2}{\pi^2}\Z{2} \right)\phi_{j,2}
    \right.\right.\nonumber\\&\left.\left.\qquad{}
            -\rat8{15}\phi_{j,1}\left( \Z{2,1} -\rat{4H^2}{\pi^2}\Z{1}
            \right)\phi_{j,1} +\rat{16}{15}\phi_{j,2}\left( 2\Z{1,2}
            +\rat{H^2}{\pi^2}\Z{2} \right)\phi_{j,2}
        \right]
    \vphantom{\frac11}\right\}
   \label{eq:strongquadc}\\&{}
    +\Ord{\varepsilon^3,\alpha^3+\gamma^3}.
    \label{eq:strongquadd}
\end{align}%
\end{subequations}

\paragraph{The model resolves nonautomomous effects, nonlinearity and inter-element interactions} The discrete model \npde~\eqref{eq:strongquad} is computed to residuals $\Ord{\varepsilon^3,\alpha^3+\gamma^3}$ and hence, supported by Potzsche and Rasmussen~\cite[Proposition~3.6]{Potzsche2006}, the model has the same order of error \cite[Theorem~5, p.213]{Roberts05c}.
The \npde{}s~\eqref{eq:strongquad} built on earlier models of nonautonomous linear diffusion~\cite[\S3.4]{Roberts06g} are recovered when one sets the nonlinearity parameter $\alpha=0$ in~\eqref{eq:strongquad}.
The truncation to errors~$\Ord{\varepsilon^3}$ ensures the model retains the interesting mean effects generated by the quadratic convolution terms parametrised by~$\varepsilon^2$ seen in the braced terms~\eqref{eq:strongquadc}.
The truncation to error~$\Ord{\alpha^3+\gamma^3}$ resolves linear dynamics within and between next nearest neighbour elements, and nonlinear dynamics within and between nearest neighbour elements.

\begin{table}
    \newcommand{\non}[1]{\textit{#1}}
    \newcommand{\rrrr}{r@{\ \ }r@{\ \ }r@{\ \ }r}
    \centering
	\caption{number of terms in the evolution $d\uj/dt=g_j(\vec U,t,\pars)$ when only three Fourier modes are used for the subgrid forced structures: the numbers in \non{italics} report the terms evident in~\eqref{eq:strongquad}. 
The left, middle and right tables give the number of terms with nonautonomous coefficient $\varepsilon^0$, $\varepsilon^1$ and $\varepsilon^2$, respectively, with, in each case, coefficients up to power three in coupling~$\gamma$ and nonlinearity~$\alpha$.
Expect many more terms when using more Fourier modes.
Blank entries are unknown.}
    \begin{tabular}{ccc}
    \begin{tabular}{r|\rrrr}
      $\varepsilon^0$  & $\gamma^0$ & $\gamma^1$ & $\gamma^2$ & $\gamma^3$ \\
        \hline
        $\alpha^3$ & 0 & 0 & & \\
        $\alpha^2$ & \non0 & 3 & 14 &\\
        $\alpha^1$ & \non0 & \non2 & 8 & 19 \\
        $\alpha^0$ & \non0 & \non3 & \non5 & 7
    \end{tabular}
    &
    \begin{tabular}{r|\rrrr}
      $\varepsilon^1$  & $\gamma^0$ & $\gamma^1$ & $\gamma^2$ & $\gamma^3$ \\
        \hline
        $\alpha^3$ & 1 & 13 & & \\
        $\alpha^2$ & \non1 & 16 & 82 &\\
        $\alpha^1$ & \non1 & \non{11} & 45 & 93 \\
        $\alpha^0$ & \non1 & \non6 & \non{10} & 14
    \end{tabular}
    &
    \begin{tabular}{r|\rrrr}
   $\varepsilon^2$       & $\gamma^0$ & $\gamma^1$ & $\gamma^2$ & $\gamma^3$ \\
        \hline
        $\alpha^3$ & 9 & & & \\
        $\alpha^2$ & \non6 & 156 & &\\
        $\alpha^1$ & \non3 & \non{42} & 238 & \\
        $\alpha^0$ & \non0 & \non0 & \non0 & 0
    \end{tabular}
    \end{tabular}
    \label{tbl:newt}
\end{table}

Computer memory~\cite{Roberts06b} currently limits us to the first few subgrid Fourier modes, $1,\sin\theta,\cos2\theta$, for this modelling of \bspde.
In modelling the microscale lattice dynamics~\eqref{eq:bssde} these three nonautonomous modes are complete, so in the application to coarse graining to~\eqref{eq:bssdm} the analysis is quite general in that it accounts for all possible microscale forcings.
Table~\ref{tbl:newt} indicates the level of complexity of the multi-parameter asymptotic series via a type of Newton diagram.
The table reports the number of terms in various parts of the discrete model \npde\ $\dot \uj=g_j(\vec U,t,\pars)$ (there are vastly more terms describing the subgrid microscale structure~$v_j(\vec U,x,t,\pars)$ of the \nsm). 
\emph{If we pursue either higher order truncations or more Fourier modes, then the complexity of the model increases combinatorially}.
Rational resolution of the subgrid scale fluctuation interactions, in order to determine their macroscale effects,  suffers from a combinatorial explosion in terms.
Thus, for the moment, truncate the model as in~\eqref{eq:strongquad}.

\paragraph{Abandon strong nonautonomous modelling}
An undesirable feature of the macroscale discrete model \npde{}s~\eqref{eq:strongquad} is the inescapable appearance in the quadratic forcing terms of fast time convolutions, such as $\Z1 \phi_{j,1} =\exp(-\beta_1t)\star \phi_{j,1}$ and $\Z{1,2} \phi_{j,2} = \exp(-\beta_1t)\star \exp(-\beta_2t)\star \phi_{j,2}$\,.
These require resolution of the subgrid fast time scales in order to maintain fidelity with the original \bspde{} and, depending on the nature of the forcing, may require incongruously small time steps for a supposedly slowly evolving model.
Such calculations become particularly laborious when the forcing involves extremely rapid fluctuations.
However, maintaining strong fidelity with the details of the full forcing~$\phi(x,t)$ is a pyrrhic victory when we are only interested in the relatively slow long term dynamics of the element amplitudes~$\uj(t)$. 
Instead, we primarily need only those parts of the quadratic forcing factors (such as $\phi_{j,0}\Z1 \phi_{j,1}$ and $\phi_{j,0}\Z{1,2} \phi_{j,2}$) that \emph{over the long macroscale time scales} emerge as a mean drift and as new forcing.
The next section develops such weak models in the context of a coarse grid discretisation.

\section{Nonautonomous resonance influences deterministic dynamics} 
\label{sec:sto}
Here we introduce a method for extracting the cumulative mean drift effects and avoiding costly resolutions of subgrid fast time scales. 
The strong model \npde~\eqref{eq:strongquad} faithfully tracks any given realisation of the original \bspde~\cite[Proposition~3.6]{Potzsche2006} whether the forcing is deterministic or stochastic; 
however,  efficient numerical simulations require some mean approximation of the fast time scales inherent in the irreducible quadratic interactions.  


Chao and Roberts~\cite{Chao95, Roberts03c, Roberts05c} discussed the case where the forcing is stochastic noise and argued that quadratic terms involving memory integrals of the noise were effectively new drift and new noise terms when viewed over long time scales  (as also noted by Drolet \& Vinal~\cite{Drolet97}).
In the strong model~\eqref{eq:strongquad} we need to understand and summarise the long term effects of the quadratic forcing terms that appear in the form $\phi_{\rho} \Z{\kappa} \phi_{\mu}$ and $\phi_{\rho} \Z{\lambda,\kappa} \phi_{\mu}$\,, returning now to the more general subscripts which define both the element and the eigenmode. 
The challenge for our macroscale discretisation of a \npde\ is to model the effect of the enormous number of interacting subgrid processes in spatially extended nonlinear \npde{}s.

We consider the long term dynamics of subgrid microscale processes $y_1$~and~$y_2$ defined via the \ode{}s
\begin{equation}
    \frac{dy_1}{dt}=\phi_{\rho}\Z{\kappa}\phi_{\mu}
    \qtq{and}
    \frac{dy_2}{dt}=\phi_{\rho}\Z{\lambda,\kappa}\phi_{\mu}\,.
    \label{eq:stoy}
\end{equation}
Summarising earlier research~\cite{Chao95, Roberts03c, Roberts05c}, first name the two convolutions that appear in the nonlinear terms~\eqref{eq:stoy} as $z_1=\Z{\kappa}\phi_{\mu}$ and $z_2=\Z{\lambda,\kappa}\phi_{\mu}$\,: then we must understand the long term properties of $y_1$~and~$y_2$ governed by the coupled system of \npde{}s
\begin{eqnarray}
    \dot y_1=z_1\phi_{\rho}\,,&&
    \dot z_1=-\beta_{\kappa} z_1 +\phi_{\mu}\,,\nonumber\\
    \dot y_2=z_2\phi_{\rho}\,,&&
    \dot z_2=-\beta_{\lambda} z_2 +z_1\,.
    \label{eq:bin}
\end{eqnarray}
One class of examples are when the autonomous forcing is simple harmonic fluctuations: say \(\phi_\mu=\cos(\omega_\mu t)\) and \(\phi_\rho=\cos(\omega_\rho t+\varphi)\).
If the frequencies are the same, \(\omega_\mu=\omega=\omega_\rho\)\,, then \(z_1=[\beta_\kappa\cos \omega t+\omega\sin \omega t-\beta_\kappa e^{-\beta_\kappa t}]/(\beta_\kappa^2+\omega^2)\).
Now \(y_1\) itself is rather complicated, but all we need is the derivative
\begin{eqnarray*}
\dot y_1&=&
\frac{\beta_\kappa\cos\varphi-\omega\sin\varphi}{2(\beta_\kappa^2+\omega^2)}
+\frac{\beta_\kappa\cos(2\omega t+\varphi)+\omega\sin(2\omega t+\varphi)}{2(\beta_\kappa^2+\omega^2)}
\\&&{}
-\frac{\beta_\kappa\cos(\omega t+\varphi)e^{-\beta_\kappa t}}{\beta_\kappa^2+\omega^2}\,:
\end{eqnarray*}
Among the enormous number of such influences in the model, the fluctuating components are generally less important than the mean drift effect, so we form a weak model by the replacement
\begin{equation}
\phi_{\rho}\Z{\kappa}\phi_{\mu}=\frac{dy_1}{dt} \mapsto 
\begin{cases}
\frac{\beta_\kappa\cos\varphi-\omega_\mu\sin\varphi}{2(\beta_\kappa^2+\omega_\mu^2)}&\omega_\mu=\omega_\rho\,,
\\0&\omega_\mu\neq\omega_\rho\,.
\end{cases}
\end{equation}
The replacement by zero when \(\omega_\mu\neq\omega_\rho\) follows because in this case all terms in~\(\dot y_1\) fluctuate and there is no mean effect.
Algebra for more convolutions is similar, just more detailed: for example, the weak modelling invokes the replacement
\begin{equation}
\phi_{\rho}\Z{\lambda,\kappa}\phi_{\mu} \mapsto
\begin{cases}
\frac{(\beta_\kappa\beta_\lambda-\omega_\mu^2)\cos\varphi
-\omega_\mu(\beta_\kappa+\beta_\lambda)\sin\varphi}
{2(\beta_\kappa^2+\omega_\mu^2)(\beta_\lambda^2+\omega_\mu^2)}
&\omega_\mu=\omega_\rho\,,
\\0&\omega_\mu\neq\omega_\rho\,.
\end{cases}
\end{equation}
Depending upon the nature of the nonautonomous effects, there will be other similar replacements to create an effective weak model that only resolves slow time scales by eliminating the fast convolutions.

The stochastic case gives another example. 
Deterministic centre manifold theory applied to the Fokker--Planck equation for the system~\eqref{eq:bin} proves~\cite[\S4]{Roberts05c} that as time $t\to\infty$ the probability density function for the \npde~\eqref{eq:bin} tends \emph{exponentially quickly} to a quasi-stationary distribution~\cite[e.g.]{Pollett90}.
The quasi-stationary distribution evolves according to a Kramers--Moyal equation which we interpret as approximating a Fokker--Planck equation for a system of \npde{}s (the neglected terms represent algebraically decaying non-Markovian effects among the $\vec y$~variables~\cite[eqn.~(11)]{Just01}).
This established analysis~\cite[\S4]{Roberts05c} of the Fokker--Planck equation for system~\eqref{eq:stoy} models the system's long-time dynamics by the stochastic \npde{}s
\begin{equation}
    \frac{dy_1}{dt}=\rat12 s+\frac{\psi_1(t)}{\sqrt{2\beta_{\kappa}}}
    \quad\text{and}\quad
    \frac{dy_2}{dt}=\frac{1}{\beta_{\kappa}+\beta_{\lambda} }\left(
    \frac{\psi_1(t)}{\sqrt{2\beta_{\kappa}}}
    +\frac{\psi_2(t)}{\sqrt{2\beta_{\lambda} }} \right),
    \label{eq:oosnn}
\end{equation}
depending upon whether the forcing terms $\phi_{\mu}$~and~$\phi_{\rho}$ are independent ($s=0$) or the same process ($s=1$).
As proved previously~\cite[Appendix~B]{Roberts05c}, and analogous to the argument of Just et al.~\cite[equation~(11)]{Just01}, the two~$\psi_i(t)$ are \emph{new forcing terms independent of $\phi_{\mu}$~and~$\phi_{\rho}$ over long time scales}.
For the case of identical $\phi_{\mu}$~and~$\phi_{\rho}$ ($s=1$) there is a mean drift~$\rat12$ in the stochastic process~$y_1$; there is no mean drift in the other case of independent $\phi_{\mu}$~and~$\phi_{\rho}$ ($s=0$).

Via the \npde{}s~\eqref{eq:oosnn} and for the specific case of \bspde{}, to obtain a weak macroscale model for \emph{long time scales} we replace the quadratic noises in the derived strong macroscale model~\eqref{eq:strongquad} as follows:\footnote{
Note that $\delta_{ij}$~and~$\delta_{mn}$, with its pair of subscripts, do \emph{not} denote a centred difference but rather denote the Dirac delta to cater for the self interaction of a noise when there is a mean drift effect ($s=1$), or not ($s=0$), as appropriate.}
\begin{eqnarray}
    \phi_{i,p}\Z{k}\phi_{j,n} &\mapsto& 
    \frac12\delta_{ij}\delta_{pn} +
    \frac{H}{k\pi\sqrt2}\psi_{ijpnk}(t)
    \,,\nonumber\\
    \phi_{i,p}\Z{l,k}\phi_{j,n} &\mapsto&
    \frac{H^3}{\pi^3(k^2+l^2)} \left[
    \frac1{l\sqrt2}\psi_{ijpnl}(t) 
    +\frac1{k\sqrt2}\psi_{ijpnlk}(t)
    \right]
    \,,\label{eq:xform}
\end{eqnarray}
where we reintroduce the element-eigenmode pair in subscripts of~$\phi$ but only require the eigenmode subscripts for~$\Z{}$ since here the decay rates $\beta_{jk}=\pi^2k^2/H^2$ are independent of element~$j$. 
The noise terms $\psi$ are effectively \emph{new
independent} white noises; that is, they are derivatives of new independent Wiener processes.

\section{Conclusion}
The critical innovation here is that we have demonstrated, via the particular example of \bspde, how  to construct a macroscale discretisation that systematically models the net effect of many subgrid microscale nonautonomous effects, both within an element and between neighbouring elements.  
The modelling applies to very general nonautonomous spatially distributed systems.

Novel theoretical support for the modelling comes from dividing the spatial domain into finite sized elements with coupling conditions~\eqref{eq:ibc}, invoking \ncm\ theory by Aulbach and Wanner~\cite{Aulbach96, Aulbach99, Aulbach2000}, Potzsche and Rasmussen~\cite{Potzsche2006}, Haragus and Iooss~\cite{Haragus2011}, and Arnold~\cite{Arnold03}, and then systematically analysing the subgrid processes together with the appropriate physical coupling between the elements.
This approach builds on success in discretely modelling autonomous \pde{}s~\cite[e.g.]{Roberts98a, Roberts00a, MacKenzie03}.
The theoretical support applies to nonautonomous difference equations on a spatial lattice, such as~\eqref{eq:bssde}, as well as to nonautonomous differential equations, such as \bspde.

The virtue of the weak modelling discussed in Section~\ref{sec:sto}, also recognised by Just et al.~\cite{Just01}, is that we may accurately take large time steps as \emph{all} the fast dynamics are eliminated in the systematic closure.
Similar ideas are employed in the study of weak nonlinear effects in water waves where reasonable approximations are obtained by substituting averages over long time scales~\cite[Chap. 10]{Dean91}.
General formulae for modelling quadratic forcing interactions~\cite{Roberts05c}, together with the iterative construction of nonautonomous slow manifold models~\cite{Roberts96a, Roberts06g}, now empower us to model a wide range of \npde{}s.
Future research may find a useful simplification of the analysis used here if it can determine the mean drift terms, quadratic in~$\varepsilon^2$, without having to compute the other quadratic forcing terms.

What about spatial domains with physical boundary conditions at their extremes? 
The artificial coupling parameter~$\gamma$ controls the information flow between adjacent elements; thus our truncation to a finite power in~$\gamma$ restricts the influence in the model of any physical boundaries to just those few elements near that physical boundary.
The approach proposed here is based upon the \emph{local} dynamics on small elements while maintaining fidelity, via \ncm\ theory, with the global dynamics of the original \npde.
In the interior of the bounded system, the methods described here remain unchanged and thus produce identical model \npde{}s.
The same methodology, but with different details, will account for physical boundaries to produce a discrete model valid across the whole domain.
Such modelling incorporating physical boundaries has already been shown for the autonomous Burgers' \pde~\cite{Roberts01b} and for shear dispersion in a channel~\cite{MacKenzie03}.

This approach to spatial discretisation of the \npdee\ may be extended to higher spatial dimensions as for autonomous \pde{}s~\cite{MacKenzie03, Roberts2011a}.
Because of the need to decompose the residuals into eigenmodes on each element, the application to higher spatial dimensions are likely to require tessellating space into simple regular elements for \npde{}s.

\paragraph{Acknowledgement} The Australian Research Council supported this research with grants DP0774311, DP0988738, DP120104260.

\bibliographystyle{plain}
\IfFileExists{ajr.sty}
{\bibliography{bibexport,ajr,bib}}
{\bibliography{bibexport}}

\end{document}